\documentclass[a4paper]{amsart}

\usepackage{amssymb,latexsym}
\usepackage[utf8x]{inputenc}
\usepackage{amsmath,amsthm}
\usepackage{amsfonts,mathrsfs}
\usepackage{mathtools}

\usepackage{halloweenmath}

\usepackage{todonotes}

\usepackage{enumerate,units}
\usepackage[all]{xy}
\usepackage{graphicx}

\usepackage{xcolor,url}
\usepackage{hyperref}
\hypersetup{
    colorlinks,
    citecolor=blue,
    filecolor=blue,
    linkcolor=black,
    urlcolor=blue
}

\usepackage{tikz-cd}

\usepackage{epigraph}

\newcommand\C{\mathbb{C}}
\newcommand\Q{\mathbb{Q}}
\newcommand\N{\mathbb{N}}
\newcommand\Z{\mathbb{Z}}
\newcommand\U{\mathcal{U}}
\newcommand\F{\mathbb{F}}
\newcommand\G{\mathbb{G}}

\newcommand\cP{\mathcal{P}}
\newcommand\PP{\mathbb{P}}
\newcommand\MM{\mathbb{M}}
\newcommand\K{\mathbb{K}}

\newcommand\CC{\mathscr{C}}
\newcommand\LL{\mathscr{L}}
\newcommand\LLf{\mathscr{L}_{\mathrm{fields}}}
\newcommand\LLr{\mathscr{L}_{\mathrm{ring}}}
\newcommand\LLm{\mathscr{L}_{\mathrm{m}}}

\newcommand{\set}[1]{\left\{ {#1} \right\}}
\newcommand{\vect}[1]{\langle {#1} \rangle}
\newcommand{\abs}[1]{\lvert {#1} \rvert}

\newcommand{\oF}[1]{\ol{\F_{{#1}}}}
\newcommand{\oT}[1]{\cl_\theta{({#1})}}

\newcommand{\NSOP}[1]{\mathrm{NSOP}_{#1}}

\newcommand{\ol}[1]{\overline{#1}}  
\newcommand{\ACFG}{\mathrm{ACFG}}
\newcommand{\ACF}{\mathrm{ACF}}

\newcommand{\ACFH}{\mathrm{ACFH}}

\newcommand{\Tm}{T_\mathrm{m}}

\newcommand{\acl}{\mathrm{acl}}
\newcommand{\dcl}{\mathrm{dcl}}
\newcommand{\cl}{\mathrm{cl}}

\newcommand{\Endd}{\mathrm{Endd}}

\newcommand{\Id}{\mathrm{Id}}

\newcommand{\Frob}{\mathrm{Frob}}

\newcommand{\alg}{\mathrm{alg}}

\newcommand{\eq}{\mathrm{eq}}
\newcommand{\pw}{\mathrm{pw}}

\newcommand{\tp}{\mathrm{tp}}

\newcommand{\dv}{\mathrm{div}}
\newcommand{\RM}{\mathrm{RM}}
\newcommand{\DM}{\mathrm{DM}}  
\newcommand{\Proj}{\mathrm{Proj}} 
\newcommand{\Diag}{\mathrm{Diag}}

\DeclareMathOperator{\fix}{fix}

\DeclareMathOperator{\car}{char}

\definecolor{airforceblue}{rgb}{0.36, 0.54, 0.66}

\theoremstyle{plain}
\newtheorem{theorem}{Theorem}[section]
\newtheorem{observation}[theorem]{Observation}
\newtheorem{corollary}[theorem]{Corollary}
\newtheorem{lemma}[theorem]{Lemma}
\newtheorem{proposition}[theorem]{Proposition}

\newtheorem{fact}[theorem]{Fact}
\newtheorem*{theorem*}{Theorem}
\newtheorem*{fact*}{Fact}

\newtheorem{alphatheorem}{Theorem}

\theoremstyle{definition}
\newtheorem{definition}[theorem]{Definition}
\newtheorem{example}[theorem]{Example}
\newtheorem{question}[theorem]{Question}

\theoremstyle{remark}
\newtheorem{remark}[theorem]{Remark}
\newtheorem{claim}{Claim}

\makeatletter
\newcommand{\setword}[2]{%
  \phantomsection
  #1\def\@currentlabel{\unexpanded{#1}}\label{#2}%
}
\makeatother

\def\seq{\subseteq}

\def\Ind{\setbox0=\hbox{$x$}\kern\wd0\hbox to 0pt{\hss$\mid$\hss}
\lower.9\ht0\hbox to 0pt{\hss$\smile$\hss}\kern\wd0}
\def\Notind{\setbox0=\hbox{$x$}\kern\wd0\hbox to 0pt{\mathchardef
\nn=12854\hss$\nn$\kern1.4\wd0\hss}\hbox to
0pt{\hss$\mid$\hss}\lower.9\ht0 \hbox to 0pt{\hss$\smile$\hss}\kern\wd0}

\def\ind{\mathop{\mathpalette\Ind{}}}

\def\indi#1{\mathop{\ \ \hbox to 0ex{\hss$\vert^{\hbox to 0ex{$\scriptstyle#1$\hss}}$\hss}
\lower1ex\hbox to 0ex{\hss$\smile$\hss}\ \ }}

\def\nindi#1{\mathop{\ \ \hbox to 0ex{\hss$\!\not{\vert}^{\hbox to 0ex{$\scriptstyle\,#1$\hss}}$\hss}
\lower1ex\hbox to 0ex{\hss$\smile$\hss}\ \ }}

\renewcommand{\models}{\vDash}

\begin{document}
\title[Generic endomorphism]{Generic multiplicative endomorphism of a field}

\author[C. d'Elb\'{e}e]{Christian d\textquoteright Elb\'ee}
\address{School of Mathematics, University of Leeds\\
Office 10.17f LS2 9JT, Leeds}
\urladdr{\href{http://choum.net/\textasciitilde chris/page\textunderscore perso/}{http://choum.net/\textasciitilde chris/page\textunderscore perso/}}
\thanks{
\begin{minipage}{0.7\textwidth}
 The author is fully supported by the UKRI Horizon Europe Guarantee Scheme, grant no EP/Y027833/1.
\end{minipage}%
\begin{minipage}{0.3\textwidth}
\begin{center}
    \includegraphics[scale=.04]{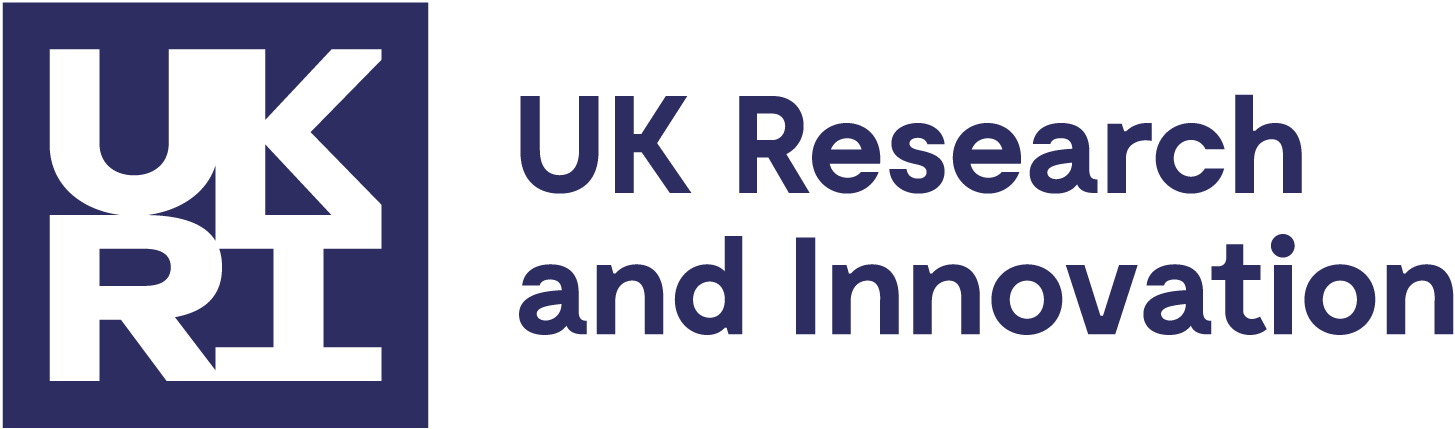}
\end{center}
\end{minipage}
}

\date{\today}

\begin{abstract}
We introduce the model-companion of the theory of fields expanded by a unary function for a multiplicative endomorphism, which we call ACFH. Among others, we prove that this theory is NSOP$_1$ and not simple, that the kernel of the map is a generic pseudo-finite abelian group. We also prove that if forking satisfies existence, then ACFH has elimination of imaginaries.
\end{abstract}

\maketitle





\setcounter{tocdepth}{1}
\hrulefill
\tableofcontents	
\hrulefill


\section*{Introduction}

This article describes the generic theory of fields equipped with a multiplicative map, that is, a map satisfying $\theta(xy) = \theta(x)\theta(y)$, $\theta(1) = 1$ and $\theta(x) = 0$ if and only if $x=0$. Such a map defines an endomorphism of the multiplicative group of the field. Formally, if $\LL$ denotes the language of fields extended by a unary function symbol $\theta$ we let $T_0$ be the $\LL$-theory of fields where $\theta$ is a multiplicative map, and let $T$ be the extension of $T_0$ expressing that the field is algebraically closed. The first main results of this paper (Corollary \ref{cor_modelcompanion}, Theorems \ref{thm_complet_diag} and \ref{thm_ACFH_NSOP1}) are summarized as follows.

\begin{alphatheorem}
The model companion of $T_0$ (and of $T$) exists, we denote it ACFH. ACFH is the model-completion of $T$. ACFH is NSOP$_1$ and not simple.
\end{alphatheorem}

 Models of ACFH, or equivalently existentially closed (e.c.) models of $T$ relate to both classical and recent classes of e.c. structures studied for their tameness properties in the (unstable) NSOP$_1$ context. As ACFH, many are obtained by an expansion process: start with a tame theory, enrich it in an expanded language and study e.c. models. Considering expansions, we distinguish between on one side, \textit{unstructured generic expansions} where no structure is imposed on the new elements in the language, such as the generic predicate \cite{CP98} or more generally Winkler's generic expansions $T^{\LL,\LL'}_{1}$ of an $\LL$-theory $T_1$ by an arbitrary $\LL'$-structure, for $\LL'\supseteq \LL$ and the generic Skolemization $T_1^\mathrm{Sk}$ expansion \cite{Win75,KrR18}. On another side, \textit{structured generic expansions} (typically of ACF) are those where the expansion is imposing some structure on new elements of the language, such as ACFA, the theory of generic difference fields \cite{CH99}. Examples include fields expanded by generic subgroups: ACFG$^+$ for the additive case in positive characteristic, ACFG$^\times$ in the multiplicative case in all characteristic \cite{dE21A,dE21B} or more recently the generalisation of those to modules, treated in the setting of positive logic \cite{dEKN21}; and the class of generic exponential fields \cite{HK21}, also treated in positive logic, which we denote \textit{Acfe} to underline that it is not elementary. 

The new theory ACFH lies at the centre of the aforementioned class of structures and provides a new sort of expansion which merges together those different examples. For instance, ACFH cumulates the ways TP$_2$ appears in previous examples. In unstructured genericity, TP$_2$ appears from the presence of a `non-linear' binary map: the theory $T_1^{\emptyset,\LL}$ has TP$_2$ as soon as $\LL$ contains a binary function symbol, see also \cite[Lemma 3.1]{Nub04}, for an analogous phenomenon in the generic Skolemization. 
As for structured genericity, the essential complexity of ACFG$^{+/\times}$ (TP$_2$) is located in the tension between the field pregeometry of ACF and the modular pregeometry of the group G$^{+/\times}$: the theory of an algebraically closed field with a predicate for a generic set (generic predicate) or a generic field (lovely pairs) is simple, however, if the predicate is a group as in ACFG$^{+/\times}$, it has TP$_2$.  ACFH cumulates those two features: TP$_2$ appears from both the presence of generic subgroups (the kernels --studied in Section \ref{section_modelsofACFH}, see also below) and the presence of non-linear binary maps (for instance the function $(x,y)\mapsto \theta(x+y)$), which behaves as a random symmetric binary map. This is characteristic of the robustness of the class of NSOP$_1$ theories to cumulating different forms of randomness, and underlines the general feeling that, if simple theories are thought of as stable ones plus random noise, NSOP$_1$ theories are those with a random roar. The presence of many generic subgroups in ACFH (see below) also leans towards this train of thought. 


\subsection*{Axioms} The axiomatization of ACFH may be regarded as in-between ACFA and \textit{Acfe}. The presence of the torsion in the multiplicative group of an algebraically closed field makes the characterisation of e.c. models of $T$ quite different from the one in \textit{Acfe}, where the domain group, $\G_a$ is torsion-free. As a matter of fact, the characterisation of e.c. models of $T$ given in Theorem \ref{thm_genhom_mult} is exactly patterned around the corresponding one for difference fields, where notions such as `projects generically' and `varieties' are replaced by corresponding notions related to the theory of the multiplicative group, that we call `projects m-generically' and `m-varieties' where `m' stands for multiplication. It relies on a characterisation of irreducible sets where the homomorphism $\theta$ can be generically extended. The characterisation of e.c. models is stated as such to emphasize that there is an underlying pattern in the construction of ACFH that might be developed further (see below). Then comes another crucial difference with \textit{Acfe}, which is that this characterisation of e.c. models of $T$ is first-order. This relies essentially on the same definability results that were used in the axiomatisation of ACFG$^\times$: the uniform definability of varieties which are \textit{free of multiplicative dependences} (or just \textit{free}), a notion that was used to axiomatize Zilber's pseudo exponential fields \cite{Zil05} and already known to be first order at that time. It was then used in different contexts, such as \cite{Tra17} (where it is called 'multiplicatively large') to axiomatize generic multiplicative circular orders on algebraically closed fields in positive characteristic and also in \cite{BHG14}. This notion actually traces back earlier, and can be related to the Mann property \cite{Man65} or other Mordell-Lang-type statements. Using Bertini arguments as in \cite{Tra17}, we deduce that the axiomatization of ACFH can be reduced to affine curves (Theorem \ref{thm_reductiontoaffinecurves}).

\subsection*{Latest development in NSOP$_1$ theories} The classes of NSOP$_1$ and NSOP$_2$ theories were defined by D\v{z}amonja and Shelah in~\cite{DS04} in order to extend downward the $(\NSOP{n})_{n\geq 3}$ hierarchy. Since the Kim-Pillay style characterisation of NSOP$_1$ theories developed by Chernikov and Ramsey in \cite{CR16}, and shortly after, a suitable notion of forking --Kim-forking-- that behaves well in this context \cite{KR20}, NSOP$_1$ theories received a considerable amount of interest from various authors, in various directions: developing further the abstract theory \cite{Ram19,CKR20,DKR22, DH21, Bos22} (even outside the first-order context \cite{HK21,DK22,Kam21}) and finding new and enlightening examples of NSOP$_1$ theories \cite{KrR18,Dob20,dE21A,dE21B,dEKN21, BdEV22}, ripe for model-theoretic treatment. Various versions of ranks have been developed in the NSOP$_1$ context (e.g. \cite{CKR20}), with more recently the promising family of local ranks developed by Dobrowolski and Hoffman \cite{DH21}, and used for answering several open questions about the theory of vector spaces with a generic bilinear form. We will study those ranks in ACFH in further work. A particularly important recent breakthrough is the proof by Mutchnik \cite{Mut22} that NSOP$_1$ theories are the same as NSOP$_2$ theories, a question already asked in \cite{DS04}. In particular, by results of Malliaris and Shelah \cite{MS17}, NSOP$_1$ theories coincide with theories which are non-maximal for the $\triangleleft^*$-order, and can also be characterized by having few higher formulas.

The theory of Kim-forking and Kim-independence have been also developed further, and the Kim-Pillay style characterisation has been considerably polished (even extended to the positive logic setting in \cite{DK22}), to have a very workable version, see Fact \ref{fact:KPnsop1}. Essentially, a theory is NSOP$_1$ if and only if there is an independence relation (over models) satisfying all the axioms of the classical Kim-Pillay theorem, with the exception of \ref{BMON}, and if so, the independence relation is Kim-independence relation over models. Further, the theory is simple if and only if Kim-independence satisfies \ref{BMON}. We show that ACFH is NSOP$_1$ and not simple using this criterion, and give a description of Kim-independence. As in ACFA, we prove that ACFH satisfies a generalised version of \ref{INDTHM}: $n$-amalgamation (Theorem \ref{thm_n-AM}). 

\subsection*{Existence or not existence, that is the question} The reason for considering independence relations \textit{over models} in the Kim-Pillay style characterisation of NSOP$_1$ theories is intrinsic to the original definition of Kim-dividing, that is, dividing with respect to an invariant Morley sequence, which only exists in general over models. On the other hand, in every known example, Kim-independence had another definition, meaningful over every set, thus a need for extending the results of \cite{KR20} to arbitrary sets became apparent. It was already known in \cite{KR20} that in an NSOP$_1$ theory, Kim-dividing (over a model) is equivalent to dividing with respect to a forking-Morley sequence, hence in \cite{DKR22}, the authors changed the definition of Kim-dividing to dividing with respect to a forking-Morley sequence, and generalized the whole theory of Kim-independence over arbitrary sets under the assumption that forking-Morley sequences exist over every set, i.e. every set is an extension base for forking, otherwise known as \textit{forking satisfies \ref{EX}}\footnote{An example of an NSOP$_1$ theory that does not satisfy the existence axiom was recently announced by Mutchnick \cite{mutchnik2024mathrmnsop1theoryexistenceaxiom}.}.

We currently do not know whether forking independence satisfies \ref{EX} in ACFH, although we strongly believe that it is true. We intend to tackle this question in a subsequent paper. However, we proved the following rather intriguing result (Theorem \ref{thm_wei}).

\begin{alphatheorem}
If forking satisfies \ref{EX} in ACFH, then ACFH eliminates imaginaries.
\end{alphatheorem}

It seems unlikely to the author that there should not exist a proof of elimination of imaginaries for ACFH that does not use \ref{EX} for forking.

\subsection*{Kernels} One of the characteristic features of ACFH is that the multiplicative map defines a family of multiplicative subgroups, the kernel of every definable endomorphism. We carry a precise study of some kernels in section \ref{section_modelsofACFH}. In a model $(K,\theta)$ of $T$, for every polynomial $P\in \Z[X]$ there is an associated endomorphism $P(\theta)$ of $K$. Thus, we identify a ring of definable multiplicative endomorphism, which we denote $\Z[\theta]$. In a model of ACFH, this ring is isomorphic to the polynomial ring $\Z[X]$ (which is of course not the case in general). We prove (Theorem \ref{thm_ker_generic}, Proposition \ref{prop_iterations}, Corollary \ref{cor_kernelspseudofiniteACFH}):
\begin{alphatheorem}
Let $(K,\theta)$ be a model of ACFH. Let $\phi\in \Z[\theta]$. Then $\ker\phi$ is a generic multiplicative subgroup of $K^\times$ and, as a pure group, $\ker\phi$ is \textit{pseudofinite-cyclic}, i.e. elementary equivalent to an ultraproduct of finite cyclic groups. Further, $(K,\theta^{(n)})$ is also a model of ACFH.
\end{alphatheorem}

 We conjecture that $\Z[\theta]$ is the whole ring of definable endomorphism in any model of ACFH.

\subsection*{Pseudofinite-cyclic groups} The ubiquity of pseudofinite-cyclic groups in ACFH is part of a more general phenomenon. The notion of pseudofinite-cyclic groups did not get much attention in the literature. In \cite{Kes14}, Kestner proves that pseudofinite abelian groups are exactly the groups where one can definably assign a notion of measure and dimension to each definable set in the sense of Macpherson-Steinhorn (MS-measurable, \cite{MS08}). In Subsection \ref{subsection_pseudofinite}, we carry out an analysis of pseudofinite-cyclic groups in our context and, using a result joint with Herzog, we prove a quite surprising criterion: for any field $K$ with $K^\times$ divisible if $\phi:K^\times\rightarrow K^\times$ is a surjective endomorphism, then $\ker\phi$ is pseudofinite-cyclic (Theorem \ref{thm_ker_surj_endo_psfc}). In particular, such endomorphisms induce an isomorphism between $K^\times/\ker\theta$ and $K^\times$. To apply this criterion in general, the difficult step is to actually exhibit a surjective multiplicative endomorphisms of a field with an infinite kernel. It turns out that they are abundant in models of ACFH, as any element of $\Z[\theta]\setminus \Z$ is surjective with infinite kernel. Pseudofiniteness of the kernel of a generic multiplicative map raises a very natural question: can a model of ACFH be obtained as a limit of `natural' multiplicative maps with finite kernels? Is the ultraproduct of models of $T$ of the form $(K_n, x\mapsto x^{k_n})$ a model of ACFH, for some choice of algebraically closed field $K_n$ and sequence $(k_n)_n$ of integers? We answer those questions in the negative in Subsection \ref{sub_natural_model}, the reason being that the kernels obtained by this process lack the genericity associated with models of ACFH. We also give some insights concerning variants of the construction more suitable to study the theory of ultraproduct of models of $T$ of the form $(K_n, x\mapsto x^{k_n})$, see Remark \ref{rk_towardtheMTofCtheta}.

\subsection*{Towards more general constructions.} The theory ACFH is the archetypical example of more general constructions that could be developed further. Let $K$ be an algebraically closed field and $G_1\seq K^n$ and $G_2\seq K^m$ denote two affine algebraic groups defined over an algebraically closed field $k_0\seq K$. Let $T_{G_1,G_2}$ be the theory in the language of fields extended by constants for $k_0$, predicates and function symbols for the groups $G_1$ and $G_2$, and $n$-ary function symbols $\theta_1,\ldots, \theta_m$ so that $\theta = (\theta_1,\ldots,\theta_m)$ defines a group homomorphism between $G_1$ and $G_2$.
\begin{question}\mbox{$~$}
\begin{enumerate}
    \item Describe e.c. models of $T_{G_1,G_2}$, do they form an elementary class?
    \item Is the class of e.c. models of $T_{G_1,G_2}$ NSOP$_1$?
\end{enumerate}
\end{question}
ACFH is the model companion of $T_{G_1,G_2}$ for $G_1 = G_2 = K^\times$ and $\textit{Acfe}$ is the class of e.c. models of $T_{G_1,G_2}$ for $G_1= K^+$ and $G_2 = K^\times$. So both questions are answered in those particular cases. One can already speculate, by following the recipe of ACFH, that the class of e.c. models of $T_{G_1,G_2}$ is elementary in the following case, for $G_2$ divisible and abelian:
\begin{enumerate}
    \item[(a)] If $G_1 = K^+$ and $\car(K) = p>0$ (note that $T_{G_1,G_2}$ may have trivial models: if $G_2 = K^\times$, see \cite{HK21});
    \item[(b)] If $G_1 = K^\times$;
    \item[(c)] If $G_1$ is an abelian variety without complex multiplication and $\car(K) = 0$.
\end{enumerate}
One has first to convince themselves that the elementarity of the class of e.c. models of $T_{G_1,G_2}$ only depends on the group $G_1$ (by checking the proof of Theorem \ref{thm_genhom_mult}), then use \cite[Theorem 5.2 $(H_4)$]{dE21A} for (a) ; Fact \ref{fact_multi_free_r_firstorder} for (b) and \cite[Fact 2]{dE21C} for (c). On the other hand, we expect a negative result if $G_1 = K^+$ if $\car(K)=0$, regardless of $G_2$, by an argument similar to the one used in \cite[Section 5.6]{dE21A} or \cite{HK21}. Nonetheless, the approach of positive model theory should yield answers to the second question in the case of a lack of model-companion. It is worth mentioning that if the existentially closed exponential fields do not form an elementary class \cite{HK21}, the existentially closed logarithmic fields ($G_1 = K^\times$ and $G_2 = K^+$) should form an elementary class.

It is the conviction of the author that, in general, the `generic subgroup' approach can be turned into a `generic endomorphism' approach. For instance, it could be considered to replace ACF with RCF in the definition of $T$. The tools from o-minimality that were used to axiomatise generic multiplicative groups in \cite{BG20} are very much likely to yield the existence of a model-companion for the expansion of RCF by a multiplicative map --let us call it RCFH. Then RCFH should inherit TP$_2$ and SOP so should be quite wild and NATP \cite{AKL21} seems to be the next best thing that one can expect for RCFH. This train of thought could lead to potentially interesting new concepts, such as dense-codense endomorphism \cite{BdEV22} or $R$-module endomorphism \cite{dEKN21}. Connections with groups with the Mann property \cite{vdDG05} are also conceivable. We see the theory ACFH as a template for a new form of generic structures.

It is rare to consider generic expansions without mentioning the machinery of interpolative fusions \cite{KTW21}. We believe that ACFH could be seen as the interpolative fusion of ACF with a hypothetic theory $T$ which would be the model-companion of the expansion of the theory $\Tm$ of the multiplicative group of an algebraically closed field by a group endomorphism, provided such a theory $T$ exists. The existence of $T$ should easily follow from an approach similar to the one taken in this paper. To our knowledge, such theory $T$ has never been studied.

\subsection*{Acknowledgements}
The author is grateful to Lotte Kestner and Guiseppina Terzo for interesting discussions. The author is very grateful to Ivo Herzog for various discussions on this project, especially concerning pseudofinite abelian groups and to Zoé Chatzidakis for her general support. Finally, the author is very grateful to the anonymous referee for very valuable comments.

\section*{Preliminaries and notations}

Let $\LL = \LLf\cup\set{\theta} $ where $\theta$ is a unary function. We define the following three $\LL$-theories
\begin{itemize}
    \item $T_0$ is the $\LL$-theory of all fields where $\theta$ is a multiplicative map;
    \item $T_1$ is the $\LL$-theory of all fields with divisible multiplicative subgroup, where $\theta$ is a multiplicative map;
    \item $T$ is the $\LL$-theory of all algebraically closed fields where $\theta$ is a multiplicative map.
\end{itemize}
A model $(K,\theta)$ of $T$ is such that $K$ is a model of ACF and $\theta(xy) = \theta(x)\theta(y)$. Then $\theta:K^\times\rightarrow K^\times$ is a multiplicative group endomorphism.

Let $\LLm = \set{\cdot,^{-1},1}$. Recall that for an algebraically closed field $K$ the $\LLm$-theory of $(K^\times, \cdot, ^{-1},1)$ is well-understood, it is obtained by adding to the theory of divisible abelian groups the description of the torsion, depending on the characteristic $p$ of $K$:
\begin{itemize}
\item if $p>0$: $\set{\exists^{=n} y \ \ y^n = 1\mid n\in \N\setminus p\N}\cup \set{\exists^{=1}y\ y^p = 1}$;
\item if $p=0$: $\set{\exists^{=n} y \ \ y^n = 1\mid n\in \N\setminus \set{0}}.$
\end{itemize}
In particular, the theory of $K^\times$ in $\LLm$ is strongly minimal. We fix an algebraically closed field $K$ and the $\LLm$-theory $\Tm$ of its multiplicative group. For a set $C\seq K^\times$, we denote by $\vect{C}$ the group generated by $C$, and $\vect{C}^\dv$ the divisible closure of $\vect C$. The operator $\vect{.}^\dv$ is the algebraic closure in the sense of $\Tm$, it defines a modular pregeometry on $K^\times$. A tuple $(a_1,\dots, a_r)\in K^n$ is called \textit{multiplicatively independent over some $C\seq K$} if for all $l_1,\ldots,l_r\in \Z$ we have $a_1^{l_1} \ldots  a_r^{l_r}\in C$ implies $l_1=\ldots = l_r = 0$. For any subgroup $C\subseteq K^\times$, the smallest $N\in \N$ such that $a^N\in C$ is called the \emph{order of $a$ over $C$}, if it exists. If there is no such $N\in \N$, the order of $a$ over $C$ is infinite. 

Note that we deal with groups which are meant to be subgroups of the multiplicative group of a field, hence we denote them multiplicatively. If $A,B,C$ are abelian groups such that $A\cap B = C$, we write $\vect{AB} = A\odot_C B$. If $C = \set{1}$, we simply write $A\odot B$.

We use the following notations, for any group $A$:
\begin{itemize}
    \item $\mu_n(A) = \set{a\in A\mid a^n = 1}$;
    \item $\mu_{p^\infty}(A) = \bigcup_n \mu_{p^n}(A)$;
    \item $\mu_\infty(A) = \bigcup_n \mu_n(A)$.
\end{itemize}
For an algebraically closed field $K$, $\mu_n(K)$ is the group of $n$-th roots of unity; $\mu_{p^\infty}$ is a $p$-Prüfer group for all $p\neq \car(K)$ and $\mu_\infty(K)$ is the group of all roots of unity. The following is standard: 
\[\mu_\infty(K) \cong \bigodot_{p\neq \car(K)} \mu_{p^\infty}(K).\]
Similarly, there exists a divisible and torsion-free (hence a $\Q$-vector space) group $V$ such that 
\[K^\times \cong  \mu_{\infty}(K) \odot V\]

An application of Zorn's lemma yields the following useful characterisation of divisible abelian groups, see e.g. Kaplanski \cite{Kap18}. We will generally refer to it as `by divisibility'. 
\begin{fact*}\label{fact_fuchs} 
Let $D$ be an abelian group, the following are equivalent.
\begin{enumerate}
    \item $D$ is divisible;
    \item for any abelian group $G\supseteq D$, there exists a subgroup $H$ of $G$ such that $ G =D\odot H$ ($D$ is a \textit{direct factor} in $G$);
    \item for any abelian groups $H\seq G$ and homomorphism $h:H\rightarrow D$ there exists a group homomorphism $h':G\rightarrow D$ such that $h'\upharpoonright H = h$.
\end{enumerate}
\end{fact*}

For any fields $E,F,K$ living in a bigger field we write $F\indi\alg_E K$ if $F$ and $K$ are algebraically independent over $E$. We write $(KF)^\alg$ for the algebraic closure of the compositum $KF$ of $K$ and $F$.

\textbf{Convention.} For any fields $F,K$, we extend the notation $\vect{F,K}$ to mean the product set, and most of the time we identify it with the group $\vect{F^\times,K^\times}$. Similarly, we sometimes use $F$ to mean $F^\times$ when it is clear from the context that we are dealing with the multiplicative group of those fields (e.g. we write $F\odot_E K$ instead of $F^\times \odot_{E^\times} K^\times$). We will also often identify the notions of multiplicative map $K\rightarrow K$ and multiplicative homomorphism $K^\times \rightarrow K^\times$.

\section{Amalgamation results for $T$}

We start by some amalgamation results for the theory $T$. 

\subsection{2-amalgamation}

\begin{lemma}\label{lm_uniqueextendoproduct}
Let $A,B,C,D$ be abelian groups such that $A\cap B = D$, and $D$ divisible. Let $\theta_A: A\rightarrow C$, $\theta_B:B\rightarrow C$ and $\theta_D:D\rightarrow C$ be group homomorphisms, such that $\theta_A$ and $\theta_B$ extend $\theta_D$, that is: $\theta_A\upharpoonright D = \theta_B\upharpoonright D = \theta_D$. Then there exists a unique homomorphism $\theta:A\odot_D B\rightarrow C$ such that $\theta$ extends $\theta_A$ and $\theta_B$. We denote it $\theta = \theta_A\odot_D\theta_B$.
\end{lemma}
\begin{proof}
As $D$ is a divisible subgroup of $A$ and $B$, $D$ is a direct factor of $A$ and $B$, hence there exists a subgroup $H_A$ of $A$ and a subgroup $H_B$ of $B$ such that $A = D\odot H_A$ and $B = D\odot H_B$. By modularity, $H_B\cap A = H_A\cap B = \set{1}$. Then $E := \vect{AB} = D\odot H_A\odot H_B$. Every element of $E$ can be written uniquely of the form $dab$ with $d\in D$, $g\in H_A$ and $h\in H_B$. Define $\theta_E(dgh) = \theta_D(d)\theta_A(g)\theta_B(h)$. By uniqueness of the decomposition of the form $dgh$, $\theta_E$ is a well-defined homomorphism of $C$. To prove uniqueness, let $\theta:E\rightarrow C$ be any homomorphism extending $\theta_A$ and $\theta_B$, then as every element has a unique expression of the form $dgh$, we have $\theta(dgh) = \theta(dg)\theta(h) = \theta_A(dg)\theta_B(h) = \theta_D(d)\theta_A(g)\theta_B(h) = \theta_E(dgh)$, so $\theta= \theta_E$.
\end{proof}

\begin{lemma}\label{lm_2aminclusion}
Let $(F,\theta_0)$ be a model of $T_1$ and $(K_1,\theta_1)$, $(K_2,\theta_2)$ be models of $T_0$ extending $(F,\theta_0)$ with $K_1\cap K_2 = F$. Let $L$ be any field extension of $K_1K_2$ such that $L^\times$ is divisible. Then there exists a multiplicative endomorphism $\theta:L^\times \rightarrow L^\times $ extending both $\theta_1$ and $\theta_2$, in other words, $(L,\theta)$ is a model of $T_1$ extending both $(K_1,\theta_1)$ and $(K_2,\theta_2)$.
\end{lemma}
\begin{proof}
First, as $K_1 \cap K_2  = F$ and as $F^\times$ is divisible, by Lemma \ref{lm_uniqueextendoproduct} there is a unique endomorphism $\theta_1\odot_F\theta_2:K_1 \odot_{F} K_2:\rightarrow K_1 \odot_{F} K_2\seq L$ extending both $\theta_1$ and $\theta_2$. As $L^\times$ is divisible, the endomorphism $\theta_1\odot_F\theta_2$ extends to and endomorphism $\theta:L^\times\rightarrow L^\times$.
\end{proof}

\begin{remark}
In the above, one can alternatively drop the assumption that $L^\times$ is divisible and assume that $(K_1,\theta_1)$ and $(K_2,\theta_2)$ are models of $T_1$. Indeed, in this case $\vect{K_1^\times,K_2^\times}$ is divisible, so it is a direct factor in $L^\times$, hence $\theta_1\odot_F \theta_2$ extends to $L^\times$.
\end{remark}

\begin{remark}[Full Existence]\label{rk_isomorphiccopy}
For any field isomorphism $\sigma: K_1\rightarrow K_2$, if $\theta_1$ is a multiplicative endomorphism of $K_1$, then $\theta_2$ defined by $\sigma\circ\theta_1\circ \sigma^{-1}$ is a multiplicative endomorphism of $K_2$. In particular, $\sigma$ becomes an $\LL$-isomorphism $\sigma: (K_1,\theta_1)\cong (K_2,\theta_2)$. It follows that for all $(K_1,\theta_1)\models T$ and $F\seq K_1$, there exists $(K_2,\theta_2)\cong_F (K_1,\theta_1)$ such that $K_1\indi\alg_F\ K_2$.
\end{remark}

\begin{lemma}[Amalgamation Property]\label{lm_2am_embeddings}
Let $(F,\theta_0),(K_1,\theta_1)$ and $(K_2,\theta_2)$ be three models of $T$ such that there exist $\LL$-embeddings $f_1:(F,\theta_0)\rightarrow (K_1,\theta_1)$ and $f_2:(F,\theta_0)\rightarrow (K_2,\theta_2)$. Then there exist a model $(L,\theta)$ of $T$ and two $\LL$-embeddings $g_1:(K_1,\theta_1)\rightarrow (L,\theta)$, $g_2:(K_2,\theta_2)\rightarrow (L,\theta)$ such that the diagram commutes.
\end{lemma}
\begin{proof}
There exists a copy $K_1'$ of $K_1$ and a field isomorphism $\sigma:K_1\rightarrow K_1'$ over $F$ such that $K_1'\indi\alg _F K_2$. Now we consider that $(K_1',\theta_1'),(K_2,\theta_2)$ and $(F,\theta_0)$ are subfields of $(K_1'K_2)^\alg$ such that $K_1'\cap K_2 = F$. We conclude by Lemma \ref{lm_2aminclusion}.
\end{proof}

\subsection{Higher amalgamation}
 We proceed to define higher amalgamation, as it was defined in \cite{HK21}, which is not the classical definition, as we may find in e.g. \cite{CH99}. See also \cite[Appendix A]{dEKN21} for a discussion on that matter and related results. We denote by $\cP(n)$ (respectively $\cP^-(n)$) the set of subsets (resp. proper subsets) of $\set{1,\ldots,n}$. 

We state the definition of $n$-amalgamation for the theory $T$ but it makes sense for any theory for which a ternary relation is defined on every model.

\begin{definition}\mbox{$~$}
\begin{itemize}
    \item Let $n\ge 3$, and let $S$ be a subset of $\cP(n)$, closed by taking subsets. Let $(F_a,\theta_a)_{a\in S}$ be a collection of models of $T$ indexed by $S$, with embeddings $\iota_{ab}:(F_a,\theta_a)\rightarrow (F_b,\theta_b)$ whenever $a\seq b$ and $a\seq b\seq c$ implies $\iota_{ac} = \iota_{bc}\circ \iota_{ab}$. We say that $(F_a,\theta_a)_{a\in S}$ is an \textit{independent $S$-system} (over $F_\emptyset$) if for every $a\seq b$,  \[\iota_{ab}(F_a)\indi{\alg}_{\iota_{ab}((\iota_{ca}(F_c))_{c\subsetneq a}} (\iota_{db}(F_d))_{a\not\subseteq d\subseteq b}.\]
    We consider that $F_a\indi{\theta}_{(F_c)_{c\subsetneq a}} (F_d)_{a\not\subseteq d\subseteq b}$, as subsets of $F_b$, where we consider every embedding $F_a\to F_b$ as an inclusion.
    \item We say that $T$ has \emph{$n$-amalgamation} ($n\ge 3$) if any independent $\mathcal{P}^-(n)$-system can be completed to an independent $\mathcal{P}(n)$-system.
\end{itemize}
\end{definition}

It is a classical fact that ACF has $n$-amalgamation with respect to algebraic independence, for each $n\geq 3$, see \cite[Proposition A.3]{dEKN21} for a complete proof. Everything one needs to know about an independent $\cP(n)$-system is contained in the following fact from Shelah.

\begin{fact}[{\cite[Fact XII.2.5]{shelah1990classification}}]\label{fact shelah}
    Let $F=(F_s)_{s\subseteq n}$ be an independent $\mathcal{P}(n)$-system of algebraically closed fields, where every $F_s$ is considered as a subset of $F_n$, and let $t\subseteq \set{1,\ldots,n}$.
    For $i<m$ let $s(i)\in \mathcal{P}(n)$ and let $\vec{a}_i\in F_{s(i)}$.
    Assume that for some formula $\phi(\vec{x}_0,\dots,\vec{x}_{m-1})$ we have $F_n\vDash \phi(\vec{a}_0,\dots,\vec{a}_{m-1})$.
    Then there are $\vec{a}_i'\in F_{s(i)\cap t}$ such that $F_n\vDash \phi(\vec{a}_0',\dots,\vec{a}_{m-1}')$, and if $s(i)\subseteq t$, then $\vec{a}_i'=\vec{a}_i$.
\end{fact}

The following is a generalised version of \cite[Lemma 5.16]{dE21A}, which was rooted in \cite{CH99}. For any $i\leq n$, we use the notations $\hat i = \set{1,\ldots,n}\setminus \set{i}$ and $\widehat{i,j} = \set{1,\ldots,n}\setminus \set{i,j}$.

\begin{lemma}\label{lm_intersection_systemACF}
Let $(F_a)_{a\in \cP(n)}$ be an independent system of algebraically closed fields. Let $a,b_1,\ldots, b_m\in \cP(n)$. Then 
\[F_{a} \cap \vect{F_{b_1 },\ldots,F_{b_m} }= \vect{F_{a\cap b_1},\ldots,F_{a\cap b_m}}.\]
In particular, we have
\[F_{\hat n} \cap \vect{F_{\hat 1 },\ldots,F_{\widehat{n-1}} }= \vect{F_{\widehat{n,1}},\ldots,F_{\widehat{n,n-1}}}.\]
\end{lemma}
\begin{proof}
 Follows from Fact \ref{fact shelah}.
\end{proof}

\begin{theorem}\label{thm_n-AM}
The theory $T$ has $n$-amalgamation, for all $n\geq 2$.
\end{theorem}
\begin{proof}
Let $(F_a,\theta_a)_{a\in \cP^-(n)}$ be an independent $\cP^-(n)$-system. As $(F_a)_{a\in \cP^-(n)}$ is an independent $\cP^-(n)$-system for ACF, there exists an algebraically closed field $F=F_{\set{1,\ldots,n}}$ such that $(F_a)_{a\in \cP(n)}$ is an independent $\cP(n)$-system for ACF. We consider embeddings as inclusions. It is enough to define a multiplicative endomorphism $\theta$ of $F$ extending simultaneously $(F_{\hat i},\theta_{\hat i})$, for $i=1,\ldots,n$. Using divisibility of $F^\times$, it is enough to extend $(F_{\hat i},\theta_{\hat i})$ to $\vect{F_{\hat 1},\ldots,F_{\hat n}}$ for all $i$. We prove it by induction on $n\geq 2$, so we assume that every independent $\cP^-(n-1)$ system can be completed to a $\cP(n-1)$-system. 
Consider $S = \set{a\in \cP^-(n), n\in a}$. The maximal elements of $S$ are $\hat 1,\ldots, \widehat{n-1}$. As $(S,\seq)$ is order-isomorphic to $\cP^{-}(n-1)$, we consider $(F_a,\theta_a)_{a\in S}$ as an independent $\cP^-(n-1)$-system over $(F_n,\theta_n)$, so in particular, there is a homomorphism $\theta_A$ of $A:= \vect{F_{\hat 1},\ldots, F_{\widehat{n-1}}}$ extending $\theta_{\hat i}$ for $i = 1,\ldots, n-1$. As $(A,\theta_A)$ extends $(F_{\widehat{n,i}}, \theta_{\widehat{n,i}})\seq (F_{\hat i},\theta_{\hat i})$, the group $C := \vect{F_{\widehat{n,1}},\ldots,F_{\widehat{n,n-1}}}\seq A$, is stable under $\theta_A$, let $\theta_C = \theta_A\upharpoonright C$. Note that $C$ is divisible. Let $(B ,\theta_B) = (F_{\hat n},\theta_{\hat n})$. As $\theta_B$ also extends each $\theta_{\widehat{n,i}}$ for each $i = 1,\ldots, n-1$, we check that $\theta_A\upharpoonright C = \theta_B\upharpoonright C = \theta_C$. By Lemma \ref{lm_intersection_systemACF}, we have $A\cap B = C$, hence by Lemma \ref{lm_uniqueextendoproduct}, $\theta_A$ and $\theta_B$ extends uniquely to $A\odot_C B$. Note that we loose uniqueness when extending $\theta_A\odot_C\theta_B$ to $F^\times$.
\end{proof}

\section{Existentially closed multiplicative endomorphism}\label{sec:ecmodels}
\subsection{Extensions of multiplicative homomorphisms}

Recall the following characterization of a divisible abelian group $D$:
\begin{center}
    for any abelian groups $H\seq G$ and homomorphism $h:H\rightarrow D$ there exists a group homomorphism $h':G\rightarrow D$ such that $h'\upharpoonright H = h$.
\end{center}


In this section, we will try to refine the property above. More precisely, in a big model of $\Tm$, assume that $a = (a_1,\ldots,a_r)$ are multiplicatively independent over $C = \vect{C}^\dv$, and let $b = (b_1,\ldots,b_t)$ be a tuple of elements from $\vect{Ca}^\dv$. If $\theta:C\rightarrow C'$ is a group homomorphism that can be extended to $\vect{Ca}^\dv$, what are the possible values of $\theta(ab)$? First, $\theta(a)$ might take any value. Each $b_i$ is of finite order say $n_i>0$ over $\vect{Ca}$ and the multiplicative type of $b$ over $\vect{Ca}$ (i.e. the type of $b$ over $Ca$ in $\LLm$) is determined by the minimal equations of elements of the form $b_1^{k_1}\ldots b_n^{k_n}$ over $\vect{Ca}$, where $k_1,\ldots,k_n$ can be choosen such that $k_i\leq n_i$. In turn this gives a finite number of equations over $C$ satisfied by the tuple $(a,b)$ which, once a multiplicatively independent tuple $a$ has been chosen, determine the multiplicative type of $b$ over $Ca$. Then any homomorphism extending $C$ on $\vect{Ca}^\dv$ sends $(a,b)$ to a tuple $(a',b')$ satisfying again those equations. Conversely, satisfying those equations is a sufficient condition for the existence of such an extension of $\theta$.

We make this more precise now. For a start, we can actually reduce those equations to an `irreducible' form, which we describe now.

\begin{definition}[Complete system of minimal equations]\label{def:completesystemofminimalequations}
Let $t,r\in \N$ and $n_1,\ldots,n_t\in \N\setminus \set{0}$, $x = (x_1,\dots,x_r)$, $y = (y_1,\dots,y_t)$ be variables. Let \[\CC = \set{(k_1,\dots,k_t)\mid 0\leq k_i\leq n_i \ \gcd(k_1,\ldots,k_t) = 1}.\] 
Let $C\subseteq K$. For each $(k_1,\dots,k_t)\in \CC$, let $N\in \N,l_1,\dots,l_r\in \Z$ (depending on $(k_1,\dots,k_t)$) be such that $\gcd(N,l_1,\dots,l_r)=1$ and $c = c_{(k_1,\dots,k_t)}\in C$. Assume further that $N = n_i$ for $(k_1,\dots,k_t) = (0,\dots,1,\dots,0)$ ($1$ at the $i$-th position).

 A \emph{complete system of minimal equations\footnote{associated to $t,r\in \N$ and $n_1,\ldots,n_t\in \N$, $x_1,\dots,x_r$, $y_1,\dots,y_t$, and $N,l_1,\dots,l_r, c$ depending on $(k_1,\dots,k_t)\in \CC$.} over $C$} is a (finite) set consisting of, for each $(k_1,\ldots,k_t)\in \CC$, one equation of the following form:
 \[(y_1^{k_1}\ldots y_t^{k_t})^N = cx_1^{l_1}\ldots x_r^{l_r}\]
 In particular it contains equations of the form $y_i^{n_i} = cx_1^{l_1}\ldots x_r^{l_r}$. The system is trivial if $n_i= 1$ for all $i\leq t$. A trivial system consist in a list of equations of the form $y_i = cx_1^{l_1}\ldots x_r^{l_r}$ for $i = 1,\ldots,t$. We identify a complete system of minimal equations with the formula over $C$ denoted $\tau(x;y)$ consisting of the conjunction of all its equations. We emphasize the fact that it depends on the separation of variables $x;y$.
\end{definition}

\begin{remark}
    If $\tau $ is a complete system of minimal equations over $C$ and $\tau(a;b)$ for some $a$ multiplicative independent over $C$, then one can think of $\tau(a,y)$ as ``isolating" the positive type of $b$ over $Ca$ in $\LLm$ (see Proposition \ref{prop_ext_mult} (2)). The number of equations needed to define a complete system of minimal equations can be reduced from $\abs{\CC}$ to $t = \abs{y}$ by considering the successive minimal equations of $b_{i+1}$ over $\vect{Cab_1,\ldots, b_i}$.
\end{remark}

 Given any complete system of minimal equations $\tau$ over a field $K$. Then rewrite each equation $(y_1^{k_1}\ldots y_t^{k_t})^N = cx_1^{l_1}\ldots x_r^{l_r}$ as $x_1^{-l_1}\ldots x_r^{-l_r}(y_1^{k_1}\ldots y_t^{k_t})^N = c$, so the realisations of $\tau$ in $K$ form what Zilber calls a \emph{shifted torus} \cite{Zil05}, which really is just a coset of an algebraic subgroup of $\G_m^n(K) = (K^{\times})^n$, for $n = r+t$.

\begin{definition}
An \textit{m-variety} (over $C$) is the set of realisation in $K$ of a complete system of minimal equations over $C$, up to permutation of variables. We often identify the m-variety with the set $\tau(z)$ of associated equations such that there exists a permutation of the variables $z$ into a partition $(x;y)$ where $\tau(x;y)$ is a complete system of minimal equations.
\end{definition}


\begin{definition}
An \textit{m-generic} of an m-variety $\tau$ over $C$ is a realisation $(a;b)$ of the associated complete system of minimal equations $\tau(x;y)$ such that $a$ is multiplicatively independent over $C$.
\end{definition}

Let $\tau(x;y)$ be a complete system of minimal equations and assume that $f : C\rightarrow C'$ is a map. Then we denote by $\tau^f(x;y)$ the conjunction of equations of the form $(y_1^{k_1}\ldots y_t^{k_t})^N = f(c)x_1^{l_1}\ldots x_r^{l_r}$, for $(y_1^{k_1}\ldots y_t^{k_t})^N = cx_1^{l_1}\ldots x_r^{l_r}$ an instance in $\tau(x;y)$. $\tau^f(x;y)$ is again a complete system of minimal equations over $C'$. If $\tau$ is the m-variety associated to $\tau(x;y)$, denote by $\tau^f$ the m-variety associated to $\tau^f(x;y)$. 

Our interest in m-variety and m-generics lies in the following proposition, whose proof is uncomplicated but lengthy. In order to avoid losing the momentum of this paper, it is given in Appendix \ref{appx_proofcms}.

\begin{proposition}\label{prop_ext_mult}
Let $C = \vect{C}^\dv \seq K^\times$ and $a\in (K^\times)^n$. 
\begin{enumerate}
    \item (Multiplicative locus) There exists an m-variety $\tau$ over $C$ such that $a$ is an m-generic of $\tau$;
    \item (Multiplicative specialization) For any m-variety $\tau$ over $C$, for any multiplicative homomorphism $\theta:C\rightarrow C'$, if $a$ is an m-generic of $\tau$ and $a'$ realises $\tau^\theta$, then there is a multiplicative homomorphism $\theta':\vect{Ca}\rightarrow \vect{Ca'}$ extending $\theta$ and such that $\theta(a) = a'$.
\end{enumerate}
\end{proposition}

\subsection{Geometric characterisation of existentially closed models}

Recall that an \emph{existentially closed model of $T$} is a model $(K,\theta)$ of $T$ such that for all models $(L,\theta')$ of $T$ which extends $(K,\theta)$, for all quantifier-free $\LL$-formulas $\varphi=\varphi(x_1,\ldots,x_n)$ with parameters in $K$, if $\varphi$ has a realisation in $(L,\theta')$ then $\varphi$ has a realisation in $(K,\theta)$. 

\begin{example}
Let $(K,\theta)$ be an existentially closed model of $T$, then $\theta$ is surjective. Let $a\in K$, it is enough to prove that there exists $b\in K$ such that $\theta(b) = a$. Let $L$ be any algebraically closed field extending $K$ non-trivially. Let $t\in L\setminus K$, then $t$ is of infinite order over the group $K^\times$, which implies that there is a homomorphism $h:t^\Z\rightarrow a^\Z\seq K^\times$. We may define a homomorphism $\theta_1: K^\times \odot t^\Z\rightarrow K^\times$ by setting $\theta_1(k t^n) = \theta(k)h(t^n) = \theta(k)a^n$. This homomorphism extends to an homomorphism $\theta': L^\times\rightarrow L^\times$ (or even $L^\times\rightarrow K^\times$) as $L^\times$ (or $K^\times)$ is divisible. The model $(L,\theta')$ of $T$ extends $(K,\theta)$, hence, as $(K,\theta)$ is existentially closed, $K\models \exists x \theta(x) = a$. 
\end{example}

By an \emph{affine variety} (over $K$) we mean an irreducible Zariski-closed subset of $K^n$ for some $n\in \N$.

\begin{definition}
Let $\tau$ be an m-variety over $K$, and let $V$ be an affine variety over $C$ such that $V\seq \tau\times K^n$. We say that \textit{$V$ projects m-generically onto $\tau$} if for all (equivalently there exists) generics $(a,b)$ of $V$, $a$ is an m-generic of $\tau$. We will also say that the projection $V\rightarrow \tau$ is \textit{generic}.
\end{definition}

\begin{theorem}\label{thm_genhom_mult}
  $(K,\theta)$ is an existentially closed model of $T$ 
  if and only if for all m-varieties $\tau\seq K^n$ over $K$ and for all affine varieties $V\seq \tau\times \tau^\theta$ which projects m-generically onto $\tau$, there exists a tuple $a\in K^n$ such that $(a,\theta(a))\in V$.
\end{theorem}

\begin{proof}
We start with a claim.
\begin{claim}\label{claim_extension_sat_tau}
    Let $(K,\theta)\models T$, $\tau\seq K^n$ an m-variety, and $V\seq \tau\times \tau^\theta$ projecting generically onto $\tau$. Then there exists an extension $(L,\theta')$ of $(K,\theta)$ and $a\in L^{n}$ such that $(a,\theta'(a))\in V$.
\end{claim}
 \begin{proof}[Proof of Claim \ref{claim_extension_sat_tau}]
Let $L$ be an elementary extension of $K$ containing a generic $(a,a')$ of $V$ over $K$. We extend $\theta : K\rightarrow K$ to $\theta^*: L \rightarrow L$ in two steps. First, as $V$ projects m-generically onto $\tau$, $a$ is an m-generic of $\tau$ over $K$ and $L\models \tau^\theta(a')$. By Proposition \ref{prop_ext_mult} (2), $\theta$ extends to $\theta' : \vect{Ka}\rightarrow \vect{Ka'}$ with $\theta(a) = a'$. Then, as $L^\times$ is divisible, we extend $\theta'$ to a endomorphism $\theta^* : L\rightarrow L$. 
\end{proof}

If $(K,\theta)$ is existentially closed, $\tau\seq K^n$, $V\seq \tau\times \tau^\theta$ projects m-generically onto $\tau$. By Claim \ref{claim_extension_sat_tau} and existential closedness of $(K,\theta)$, there exists $a\in K^{n}$ such that $(a,\theta(a))\in V$.

  Conversely, let $(K,\theta)$ be a model of $T$ that satisfies the right hand side condition. In order to show that $(K,\theta)$ is existentially closed, one has to show that every finite system of equations and inequations of $\LL$-terms with parameters in $K$ that has a solution in an extension of $(K,\theta)$ has a solution in $K$. As $T\models \forall xy ( x\neq y \leftrightarrow \exists z z(x -y) = 1)$ it is sufficient to consider only systems of equations of the form $t(x) = 0$, for $t(x)$ an $\LL$-term.

\begin{claim}[Linearisation]\label{claim_linearisation_mult}
    For every model $(K,\theta)$ of $T$, for every $\LL$-term $t(x)$, there is a cunjunction of $\LLf$-equations $\varphi(z,z')$ such that $t(x) = 0$ has a solution in $K$ if and only if $\varphi(z,\theta(z))$ has a solution in $K$, for some tuple $z$.
\end{claim}

 \begin{proof}[Proof of Claim \ref{claim_linearisation_mult}]
    Assume that there is an occurence of $\theta(t'(x))$ in $t(x)$, for some $\LL$-terms $t'$ and that $\theta$ does not appear in $t'$. Then $t(x) = 0$ has a solution in $K$ if and only if the formula (with free variable $z$)
    $$\exists y \exists x (z = t'(x) \wedge y = \theta(z)\wedge \tilde{t}(x,y) = 0)$$
    has a solution in $K$ ; where $y$ is a new variable, $z$ is a single variable and $\tilde t(x,y) = t[\theta(t'(x))/y](x)$ is the term obtained by replacing the mentioned occurence of $\theta(t'(x))$ in $t(x)$ by $y$. Now there is one less occurence of $\theta$ in $\tilde t(x,y)$ than in $t(x)$, and there are no occurences of $\theta$ in $t'$. By reiterating the same operation inside $\tilde t(x,y)$ one get that there is an $\LL$-formula $\psi(z)$ in the language $\LL$ with $z = (z_1,\dots,z_r)$ such that :
    \begin{itemize}
      \item $t(x) = 0$ has a solution in $K$ if and only if $\psi(z)$ has a solution in $K$;
      \item every occurence of $\theta$ in $\psi(z)$ is of the form $\dots \wedge v = \theta(z_i)\wedge\dots$, for some bound variable $v$ (in particular there is no iteration of $\theta$ in $\psi(z)$).
    \end{itemize}
    Now replace each occurence of $\theta(z_i)$ in $\psi(z)$ by a new single free variable $z'_i$ and call $\psi(z,z')$ the resulting formula, which is in the language $\LLr$. Then $t(x) = 0$ has a solution in $K$ if and only if $\psi(z,\theta(z))$ has a solution in $K$. Finally, we may also assume that there are no more existential quantifiers in $\psi(z,z')$ by making bound variables free and adding them to $z,z'$. 
\end{proof}

 Let $(t_i(x))_i$ be $\LL$-terms with parameters in $K$ and assume that $\psi(x)$ is the formula $\bigwedge_i t_i(x) = 0$. Assume that $\psi(x)$ has a solution in an extension of $(K,\theta)$. Using Claim~\ref{claim_linearisation_mult}, there exists an $\LLr$-formula $\varphi(z,z')$ consisting of equations, also with parameters in $K$ such that $\varphi(z,\theta(z))$ has a solution in that extension, say $a = (a_1,\ldots,a_n)$, for $n=\abs{z}$. We may assume that $a\cap K = \emptyset$, in particular $a_i\neq 0$ for all $i<n$.
 
By Proposition \ref{prop_ext_mult} (1), let $\tau(z)$ be an m-variety over $K$ such that $a$ is an m-generic of $\tau$. We also have $\theta(a)\in \tau^\theta$. Let $V\subset K^{n}\times K^n$ be the locus of $(a,\theta(a))$ over $K$. As $\varphi(z,z')$ is a conjunction of equations, $\varphi(K)$ is a Zariski-closed set so we have $V(K)\seq \varphi(K)$. For the same reason, $V\seq \tau\times \tau^\theta$. By Proposition \ref{prop_ext_mult} (1), $a$ is an m-generic of $\tau$ over $K$, as $(a,\theta(a))$ is a generic of $V$ we have that $V$ projects m-generically onto $\tau$. By hypotheses there exists $(b,\theta(b))\in V(K)$, hence $(b,\theta(b))\models \varphi(z,z')$.
\end{proof}

\section{Axiomatisation and model-completeness}\label{section_axiomandmodelcompleteness}

\subsection{The geometric characterisation is first-order}\label{subsection_geometricchar} From Theorem \ref{thm_genhom_mult}, it is clear that  in order to axiomatise the class of existentially closed models of $T$, one needs to express ``$V\seq \tau\times \tau^\theta$ projects m-generically onto $\tau$" in a first-order way.

Fix an algebraically closed $K$ and an algebraically closed extension $\K$ of infinite transcendence degree over $K$.

The following notion is due to Zilber \cite{Zil05}. It also appears later in \cite{BHG14}, and more recently in \cite{Tra17}.
\begin{definition}
  Let $r,m\in \N$. An affine variety $V\seq K^{r+m}$ is \emph{multiplicatively $r$-free} if for all (equivalently there exists) generics $(a_1,\ldots,a_{r+m})$ of $V$ over $K$, $a_1,\ldots,a_r$ is multiplicatively independent over $K$.
\end{definition}

\begin{lemma}
Let $\tau\seq K^n$ be an m-variety over $K$. Let $n = r+t$ and $\tau(x;y)$ the complete system of minimal equations associated to $\tau$, with $\abs{x} = r$ and $\abs{y} = t$. Let $V$ be an affine variety $V\seq \tau\times K^m\seq K^r\times K^t\times K^m$. The following are equivalent:
\begin{enumerate}
    \item $V$ projects m-generically onto $\tau$;
    \item $V$ is multiplicatively $r$-free.
\end{enumerate}
\end{lemma}
\begin{proof}
This is by definition.
\end{proof}

With this in mind, we can re-state the main theorem of the previous section:
\begin{theorem}\label{thm_genhom_mult_version_axioms}
  $(K,\theta)$ is an existentially closed model of $T$ 
  if and only if for all complete systems of minimal equations $\tau(x;y)$ over $K$ with $\abs{x} = r$ and $\abs{y} = t$ and for all multiplicatively $r$-free affine varieties $V\subset \tau(K)\times \tau^\theta(K)\seq K^r\times K^t\times K^r\times K^t$ there exist tuples $a\in K^r,b\in K^t$ such that $(ab,\theta(ab))\in V$.
\end{theorem}

We extend the notion of being multiplicatively $r$-free to all definable sets, in the most natural way. 
\begin{definition}
A definable set $X\seq K^{r+t}$ is \emph{multiplicatively $r$-free} if one of the irreducible component of its Zariski-closure is multiplicatively $r$-free. Equivalently, the $r$ first coordinates of some generic of $X$ over $K$ are multiplicatively independent over $K$. 
\end{definition}

\begin{lemma}\label{lm_char_r-free}
 Let $X\seq K^{r+t}$ be a definable set. The following are equivalent:
\begin{enumerate}
    \item $X$ is multiplicatively $r$-free;
    \item there is $a\in X(\K)$ such that $a_1,\dots,a_r$ are multiplicatively independent over $K$;
    \item for all finite sets $S$ of multiplicative equations $x_1^{k_1}\ldots x_r^{k_r} = c$ with $c\in K$, there exists $a\in X(K)$ such that $(a_1,\dots,a_r)$ does not satisfy any equations of $S$;
    \item $Y = \Proj(X)\seq K^r$ is multiplicatively $r$-free.
\end{enumerate}
\end{lemma}
\begin{proof}
This is an easy checking, from the definitions, and compactness.
\end{proof}

Let $V\subseteq K^r\times K^m$ an algebraic variety. The condition `$V\subseteq K^r\times K^m$ is multiplicatively $r$-free' is a first-order condition, this result first appears, to our knowledge, in \cite[Theorem 3.2]{Zil05}, where it is used to axiomatize the pseudo-exponentiation. $V$ is multiplicatively $r$-free corresponds to \textit{$pr_xV$ is free of multiplicative dependencies}, in Zilber's terms. It appears then in various places, such as \cite{KZ14}, \cite{BHG14} and more recently in \cite{Tra17}. For $m=0$ is it also called `free' or `multiplicatively large'. We recall the main ingredients of the proof that it is a first-order condition, based on the presentation in \cite[Section 3]{Tra17}. 

\begin{fact}\label{fact_ingredients}
Let $V\subseteq K^n$ be an affine variety.
\begin{enumerate}
    \item $V$ is multiplicatively $n$-free if and only if for all $c\in K$, $(k_1,\ldots,k_n)\in \Z^n\setminus \set{(0,...,0)}$, $V$ is not included in the zero set of the equation $x_1^{k_1}\ldots x_n^{k_n} = c$. 
    \item Every definable subgroup of $(K^\times)^n$ is defined by a finite set of equations of the form $x_1^{k_1}\ldots x_n^{k_n} = 1$ for some $(k_1,\ldots,k_n)\in \Z^n$.  
    \item (Zilber's Indecomposable) If $(1,\dots,1)\in V$ then $\Pi_{2n}(V\cap (K^{\times})^n)$ is the smallest (for inclusion) definable subgroup of $(K^\times)^n$ containing $V\cap (K^\times)^n$, where $\Pi_n(X) = \set{x_1\ldots x_n\mid x_i\in X}$.
    \item If $(1,\ldots , 1)\in V$, then $V\cap (K^\times)^n$ is multiplicatively $n$-free if and only if $\Pi_{2n}(V\cap (K^\times)^n) = (K^\times)^n$. 
\end{enumerate}
\end{fact}
\begin{proof}
\textit{(1)} Is easy by compactness and irreducibility of $V$. \textit{(2)} This follows from \cite[Lemma 7.4.9]{Mar02} and \cite[Corollary 3.2.15]{BG06}. \textit{(3)} See \cite[Theorem 7.3.2]{Mar02}. \textit{(4)} Follows from \textit{(1)}, \textit{(2)} and \textit{(3)}.
\end{proof}

\begin{remark}\label{rk_ingredients}
If $V\seq K^{r}$ is an affine variety and $O\seq K^{r}$ is a Zariski open set, then $V$ is multiplicatively $r$-free if and only if $V\cap O$ is multiplicatively $r$-free. In particular, $V$ is multiplicatively $r$-free if and only if $V\cap (K^\times)^r$ (which is always nonempty if $V$ is multiplicatively $r$-free) is multiplicatively $r$-free. Further, for any $c\in (V\cap K^{\times})^r$, $V$ is multiplicativly $r$-free if and only if $c^{-1}(V\cap K^{\times})^r$ is multiplicatively $r$-free.
\end{remark}

We need some classical definability results in algebraically closed fields \cite{VDD78,Joh16}, this is \cite[Fact 3.18]{Tra17}.

\begin{fact}\label{fact_uniformdefzar}
Let $\psi(x,y)$ be an $\LLr$-formula. Then there exist $\LLr$-formulas $\rho(x,z)$, $\delta_d(y),\mu_r(y),\iota(y)$ (depending on $\varphi$) such that for all $b\in K^{\abs{y}}$
\begin{enumerate}
    \item $K\models \delta_d(b)$ if and only if $\dim(\varphi(K,b))=d$;
    \item $K\models \mu_r(b)$ if and only if the Morley degree $\DM(\varphi(K,b))$ equals $r$;
    \item $K\models \iota(b)$ if and only if the Zariski closure of $\varphi(K,b))$ is an affine variety;
    \item $V$ is an irreducible component of the Zariski closure of $\psi(K,b)$ if and only if there exists $c\in K^{\abs{z}}$ such that $\rho(K,c) = V$.
\end{enumerate}
\end{fact}

We obtain:

\begin{fact}\label{fact_multi_free_r_firstorder}
Let $\varphi(x,x',y)$ be any $\LLr$-formula, and $r = \abs{x}, t = \abs{x'}$. Then there exists a formula $\delta_\varphi^x(y)$ such that $K\models \delta_\varphi^x(b)$ if and only if the set $X\seq K^{r+t}$ defined by $\varphi(x,x',b)$ is multiplicatively $r$-free.
\end{fact}
\begin{proof}
Let $\psi(x,y) = \exists x' \varphi(x,x',y)$. By Lemma \ref{lm_char_r-free} (4), it is enough to show that ``$\psi(x,y)$ defines a multiplicatively $r$-free set" is definable in $y$. Let $b\in K^y$, by Fact \ref{fact_ingredients} and Remark \ref{rk_ingredients}, if $V$ is an irreducible component of the Zariski closure of $\psi(K,b)$, then $V$ is multiplicatively large if and only if $\Pi_{2r}(a^{-1} (V\cap (K^{\times})^{r}) = (K^{\times})^{r}$, for some $a\in V\cap (K^{\times})^{r}$. Let $\rho(x,z)$ be as in Fact~\ref{fact_uniformdefzar}. Let $\delta_\varphi^x$ be a formula expressing the following:
\[``\exists c [\exists a\in \rho(K,c)\cap (K^{\times})^{r}]  \wedge [\Pi_{2r}(a^{-1} (\rho(K,c)\cap (K^{\times})^{r}) = (K^{\times})^{r}]"\]
\end{proof}
 
An easy compactness argument yields the following analogue of \cite[Corollary 3.20]{Tra17}.

\begin{corollary}\label{cor_boundwith11}
Let $\varphi(x,x',y)$ be any $\LLr$-formula, and $r = \abs{x}, t = \abs{x'}$. Assume that for all $b\in K^y$, $V_b = \varphi(K,b)\seq K^{r+t}$ defines an affine variety such that $(1,\ldots,1)\in W_b $ where $W_b$ is the Zariski closure of $\Proj(V_b)\seq K^r$. Then there exists a finite set $F\seq \Z^{r}$ such that for all $b\in K^{y}$, either $W_b$ is included in the zero set of $x_1^{k_1}\ldots x_r^{k_r} = 1$ for some $(k_1,\ldots,k_r)\in F$
or $V_b$ is multiplicatively $r$-free.
\end{corollary}
\begin{proof}
This follows from Fact \ref{fact_ingredients} (1) and Fact \ref{fact_multi_free_r_firstorder}, as the Zariski closure of the projection of a variety is again a variety. As $(1,\ldots,1)\in W_b$, Fact \ref{fact_ingredients} (1) applies with $c$ is equal to $1$.
\end{proof}




\subsection{Reduction to affine curves}\label{subsec:reductiontoaffinecurves}

We now show that the geometric characterisation given in Theorem \ref{thm_genhom_mult} can be reduced to the case where $V\seq K^{2r+2t}$ is a multiplicatively $r$-free \emph{curve}, i.e. of dimension $1$. This section is basically a rewriting of \cite[Section 4.1]{Tra17}, where we adapt the argument to include the case where $V\seq K^{r+t}$ is not only multiplicatively $r+t$-free, but $r$-free.

 For any $b\in K^n\setminus \set{(0,\ldots,0)}$ we denote by $H_b$ the affine hyperplane defined by the equation \[b_1x_1+\ldots+b_nx_n = 1.\] 
For any $c\in K^n\setminus \set{(0,\ldots,0)}$, let $S_c$ be the set of $b\in K^n\setminus \set{(0,\ldots,0)}$ such that $c\in H_b$. Note that $S_c = H_c$ for all $c$. The following is \cite[Lemma 4.3]{Tra17}, and uses Bertini's theorem.

\begin{fact}\label{fact_tran_bertini}
Let $V\seq K^n$ be an affine variety of dimension $m+1$. Then there is $c\in V$ such that the set $Y_c$ of tuples $(b_1,\ldots,b_m)\in (K^n\setminus \set{(0,\ldots,0)})^m$ such that the closed set
\[V\cap H_{b_1}\cap \ldots \cap H_{b_m}\]
is of Morley degree $1$, of dimension $1$ and $c$ belongs to its maximal component, is Zariski dense in $S_c^m$. If in addition $X\seq V$ is definable with $\dim X < \dim V$, then the set of $(b_1,\ldots,b_m)\in Y_c$ such that $X\cap H_{b_1}\cap \ldots \cap H_{b_m}$ is of dimension $0$ is also Zariski dense in $S_c^m$.
\end{fact}

\begin{lemma}
Let $n = r+t$ and $V\seq (K^\times)^{r}\times K^t$ be a multiplicatively $r$-free affine variety of dimension $\RM(V) = m+1$. Then there exists $(b_1,\ldots,b_m)\in (K^n\setminus \set{(0,\ldots,0)})^m$ such that \[W = V\cap H_{b_1}\cap \ldots \cap H_{b_m}\]
is such that $\dim(W) = \DM(W) = 1$ and such that its (unique) irreducible component of dimension $1$ is a multiplicatively $r$-free affine curve.
\end{lemma}
\begin{proof}
Let $c\in V$ be as in Fact \ref{fact_tran_bertini}. For all $b = (b_1,\ldots,b_m)\in Y_c$, the Zariski closed set $W_b = V\cap H_{b_1}\cap \ldots \cap H_{b_m}$ is an affine curve containing $c$. Let $c' = (c_1^{-1},\ldots,c_r^{-1}, c_{r+1},\ldots,c_{r+t})$. By considering $V' = c'V$ (multiplication coordinatewise), and changing $c$ to $(1,\ldots,1,c_{r+1}^2,\ldots,c_{r+t}^2)$, we may assume that $(1,\ldots,1)$ belongs to $\Proj(W_b)\seq K^r$, for any $b\in Y_c$. By Fact \ref{fact_uniformdefzar}, the set $Y_c$ is definable over $c$. Also by Fact \ref{fact_uniformdefzar}, there exists an $\LLr$-formula $\varphi(x,x',z)$ such that for all $b\in K^{nm}$, $\varphi(K,b)\seq K^r\times K^t$ is empty if $b\notin Y_c$ and equals the (unique) maximal component of $W_b$ for $b\in Y_c$. 
Let $F\seq \Z^r$ be a finite set of size $N$ as in Corollary \ref{cor_boundwith11}, and let $U\seq K^r$ be the set of realisations of the formula $\bigvee_{k=(k_1,\ldots,k_r)\in F} x_1^{k_1}\ldots x_r^{k_r} = 1$ and $X = (U\times K^t)\cap V$. Then for all $b\in Y_c$, $W_b\seq X$ if and only if $W_b$ is not multiplicatively $r$-free. Note that $\dim X < \dim V$ because $V$ is multiplicatively $r$-free, so by the second part of Fact \ref{fact_tran_bertini}, the set of $b\in Y_b$ such that $X\cap H_{b_1}\cap \ldots \cap H_{b_m}$ is of dimension $0$ is Zariski dense in $Y_b$. So for $b$ in this set, $W_b $ is not included in $X\cap H_{b_1}\cap \ldots \cap H_{b_m}$, so $W_b$ is not included in $X$, so $W_b$ is multiplicatively $r$-free.
\end{proof}

\begin{theorem}\label{thm_reductiontoaffinecurves}
  $(K,\theta)$ is an existentially closed model of $T$ if and only if for all m-varieties $\tau\seq K^n$ over $K$ and for all affine curves $C\seq \tau\times \tau^\theta$ which projects m-generically onto $\tau$, there exists a tuple $a$ from $K^n$ such that $(a,\theta(a))\in C$.
\end{theorem}

\subsection{ACFH, completions and types}

\begin{definition}
Let ACFH be the $\LL$-theory expanding $T$ and expressing the geometric axioms (Theorem \ref{thm_genhom_mult}): for all formulas $\tau(x,y,z)$ and $\varphi(x,y,x',y',z')$ with $\abs{x} = \abs{x'}$ and $\abs{y} = \abs{y'}$ such that for any $c$, $d$, $\tau(x;y,c)$ is a complete system of minimal equations and $\varphi(x,y,x',y',d)$ defines a multiplicatively $\abs{x}$-free variety, if \[\varphi(x,y,x',y',d)\rightarrow \tau(x,y,c)\wedge \tau(x',y',\theta(c)),\] then $\varphi(x,y,\theta(x),\theta(y),c)$ is consistent.
\end{definition}

\begin{corollary}\label{cor_modelcomplete}
$(K,\theta)\models \ACFH$ if and only if $(K,\theta)$ is an existentially closed model of $T$. In particular, ACFH is model-complete.
\end{corollary}

\begin{corollary}\label{cor_modelcompanion}
The theory ACFH is the model-companion of $T$, $T_1$ and $T_0$.
\end{corollary}
\begin{proof}
Every model of ACFH is a model of $T$, every model of $T$ is a model of $T_1$ and every model of $T_1$ is a model of $T_0$. Observe first that every model of $T_0$ extends to a model of $T$ and hence of $T_1$. Indeed, let $(F,\theta)$ be a model of $T_0$ and $K$ be the algebraic closure of $F$. As $K^\times $ is divisible, the endomorphism $\theta: F^\times \rightarrow F^\times\seq K^\times$ extends to an endomorphim $\theta': K^\times \rightarrow K^\times$. It remains to show that every model $(F,\theta)$ of $T$ has an extension which is a model of ACFH, which follows from Claim \ref{claim_extension_sat_tau} (in the proof of Theorem \ref{thm_genhom_mult}) and a classical chain argument.
\end{proof}

Let $\PP$ be the set of prime numbers.
\begin{corollary}[Uniformity]
For each prime number $p$ let $K_p$ be a model of ACFH of characteristic $p$. Then for any non-principal ultrafilter $\U$ on $\PP$, the ultraproduct $\Pi_\U K_p$ is a model of ACFH of characteristic $0$.
\end{corollary}
\begin{proof}
The formula expressing the condition ``$\varphi(x,x',y)$ defines a multiplicatively $r$-free affine variety" does not depend on the characteristic.
\end{proof}

Corollary \ref{cor_modelcompanion} and Lemma \ref{lm_2am_embeddings} yield the following (see e.g. \cite[8.4, Exercise 9]{Hodges}):

\begin{theorem}\label{thm_complet_diag}
Let $(F,\theta)\models T$ and $(K_1,\theta_1), (K_2,\theta_2)$ be two models of ACFH extending $(F,\theta)$, then 
\[(K_1,\theta_1)\equiv_{F} (K_2,\theta_2).\]
In other words for any $(F,\theta)\models T$, the theory $\ACFH\cup \Diag(F)$ is complete.
\end{theorem}


\begin{remark}
The theory ACFH is the \textit{model-completion} of $T$ \cite[8.4, Exercise 9]{Hodges}: $T'$ is the \textit{model-completion of $T$} if $T'$ is the model-companion of $T$ and $T$ has the amalgamation property. This is equivalent to $T'$ is the model-companion of $T$ and for all $M\models T$, the theory $T\cup \Diag(M)$ is complete. 
Note that if the model-companion of $T$ only depends on $(T)_\forall$, the model-completion does depend on $T$. ACFH is the model-companion of $T_0$, but it is not its model-completion: $\Id: \Q^\times \rightarrow \Q^\times$ has many extensions to $\Q^\alg$, each defining a different completion of $\ACFH \cup \Diag(\Q,\Id)$, so it is not complete. 
\end{remark}

\begin{remark}\label{remark_Fpbarclosed}
Let $(K,\theta)$ be a model of $T$ of characteristic $p>0$. As every element in $\oF p\setminus \set{0}$ is a root of $1$, the field $\oF p\seq K$ is closed under $\theta$, so $(\oF p,\theta\upharpoonright_{\oF p})\models T$. 
\end{remark}

\begin{corollary}[Completions of ACFH, positive characteristic]\label{cor:completionspositivecharacteristic}
Let $(K_1,\theta_1)$ and $(K_2,\theta_2)$ be two models of $\ACFH$ of characteristic $p>0$. Then
\[(K_1,\theta_1)\equiv (K_2,\theta_2)\iff  (\oF p,\theta_1\upharpoonright _{\oF p})\cong (\oF p,\theta_2\upharpoonright _{\oF p}).\]
\end{corollary}
\begin{proof}
From left to right, there exists elementary extensions of $(K_1,\theta_1)$ and $(K_2,\theta_2)$ which are isomorphic. The right to left direction follows from Theorem \ref{thm_complet_diag} and Remark \ref{remark_Fpbarclosed}.
\end{proof}
In particular, each positive prime $p$ and $n\in \N$ defines a unique completion of ACFH given by $(\oF p, x\mapsto x^n)$, but there are many more endomorphisms of $\oF p ^\times$ that define a completion of ACFH.

\begin{definition}
Let $(F,\theta)\models T$ and $A\seq F$. We define the following operators $\cl_\theta^n(A)$ and $\cl_\theta(A)$ inductively as follows:
\begin{itemize}
    \item $\cl_\theta^0(A) = \ol{A}$;
    \item $\cl^{n+1}_\theta(A) = \ol{\cl_\theta^n(A)\cup \theta(\cl_\theta^n(A))}$;
    \item $\cl_\theta(A) = \bigcup_{n\in \N} \cl_\theta^n(A)$
\end{itemize}
We call $\cl_\theta(A)$ the \textit{$\theta$-closure of $A$}. 
\end{definition}
It is easy to check that for any $A\seq (F,\theta)\models T$, $\cl_\theta(A)$ is algebraically closed as a field, and is closed under $\theta$, hence $(\cl_\theta(A),\theta\upharpoonright_{\cl_\theta(A)})\models T$. We denote $\tp_\theta(A)$ the $\LL$-type of $A$, and $\acl_\theta(A)$ the $\LL$-algebraic closure of $A$. 

\begin{proposition}\label{prop:typesACFH}
Let $(K,\theta)\models \ACFH$ and $A,B\subseteq K$. Then $\tp_\theta(A) = \tp_\theta(B)$ if and only if there is an $\LL$-isomorphism 
\[\sigma : (\cl_\theta(A), \theta\upharpoonright_{\cl_\theta(A)}) \cong (\cl_\theta(B), \theta\upharpoonright_{\cl_\theta(B)})\]
such that $\sigma(A) = B$.
Furthermore, $\acl_\theta(A) = \cl_\theta(A)$.
\end{proposition}
\begin{proof}
We may assume that $(K,\theta)$ is sufficiently saturated. The left to right implication is standard, as $\cl_\theta(A)$ is closed under $\theta$. Conversely, assume that $\sigma : \cl_\theta(A)\rightarrow  \cl_\theta(B)$ is such that $\sigma(A) = B$. As $\sigma$ is a field isomorphism, it extends to a field automorphism of $K$. Let $\theta' = \sigma^{-1}\circ\theta\circ \sigma$. Then $\theta'\upharpoonright_{\cl_\theta(A)} = \theta\upharpoonright_{\cl_\theta(A)}$. By Theorem \ref{thm_complet_diag}, $(K,\theta)\equiv_{\cl_\theta(A)} (K,\theta')$. This implies that there is a field automorphism $\sigma'$ of $K$ over $\cl_\theta(A)$ which is an $\LL$-isomorphism $\sigma':(K,\theta)\cong (K,\theta')$, so $\theta = \sigma'^{-1}\circ \theta'\circ\sigma'$. Now $\sigma\circ\sigma'$ is an $\LL$-automorphism of $K$ extending $\sigma\upharpoonright_{\cl_\theta(A)}$, so $\tp_\theta(A) = \tp_\theta(B)$.

It is clear that $\cl_\theta(A)\seq \acl_\theta(A)$. For the other direction, assume that $A = \cl_\theta(A)$ and $b\notin A$. Let $B = \cl_\theta(bA)$. By Remark \ref{rk_isomorphiccopy}, there exists $B'$ and an endomorphism $\theta'$ of $B'$ such that we have $\sigma:(B',\theta')\cong_A (B,\theta)$ with $B'\indi{\alg}_A\ K$. Let $b'\in B'$ be the preimage of $b$ by the $\LL$-isomorphism $\sigma$ over $A$. We have $B' = \cl_{\theta'}(Ab')$, and $b\neq b'$. By Lemma \ref{lm_2aminclusion}, there exists a model $(L,\theta_L)$ of ACFH extending both $(B',\theta')$ and $(K,\theta)$. Applying the previous result (in $(L,\theta_L)$), we conclude that that $\tp_\theta(b/A) = \tp_\theta(b'/A)$, and $b'\neq b$. By reiterating we may construct unboundedly many realisations of $\tp_\theta(b/A)$, hence $b\notin \acl_\theta(A)$.
\end{proof}

\begin{corollary}[Completions of ACFH, characteristic $0$]
Let $(K_1,\theta_1)$ and $(K_2,\theta_2)$ be two models of ACFH of characteristic $0$. Then
\[(K_1,\theta_1)\equiv (K_2,\theta_2)\iff  (\cl_\theta(\Q),\theta_1\upharpoonright _{\cl_\theta(\Q)})\cong (\cl_\theta(\Q),\theta_2\upharpoonright _{\cl_\theta(\Q)}).\]
\end{corollary}
\begin{proof}
The proof is as in Corollary \ref{cor:completionspositivecharacteristic} where the use of Remark \ref{remark_Fpbarclosed} is replaced by the fact that $\cl_\theta(\Q)$ is closed under the endomorphism.
\end{proof}

Again, any power function $x\mapsto x^n$ defines a unique completion on $\Q^\alg$. However, completions of ACFH are not necessarily given by a multiplicative endomorphism of $\Q^\alg$. For instance, there is a completion where $\theta(2)$ is a transcendental element over $\Q$. This is clearly a consistent type: if $t$ is transcendental over $\Q$, then $(2,t)$ is a multiplicatively independent tuple so there is an endomorphism of $\Q(t)^\alg$ sending $2$ to $t$. This endomorphism does not restrict to $\Q^\alg$. This is the main difference between the characteristic $0$ and the positive characteristic: the multiplicative degree of $\oF p$ is $0$ and the multiplicative degree of $\Q^\alg$ is $\aleph_0$ (the prime numbers form a multiplicatively independent set).


\section{Kim-independence and NSOP$_1$}
\subsection{Preliminaries}

Let $T$ be a complete theory with monster model $\MM$.

\begin{definition}
 Let $\ind$ be an invariant ternary relation on small subsets of $\MM$. We define the following axioms.
\begin{enumerate}[$(1)$]
\item (\setword{normality}{NOR}) If $A\ind_C B$ then $AC\ind_C B$.
\item (\setword{monotonicity}{MON}) If $A\ind_C BD$ then $A\ind_C B$.
\item (\setword{base monotonicity}{BMON}) If $A\ind_C BD$ then $A\ind_{CD} B$.
\item (\setword{finite character}{FIN}) If $a\ind_C B$ for all finite $a\seq A$, then $A\ind_C B$.
  \item (\setword{existence}{EX}) $A\ind_C C$ for any $A$ and $C$.
  \item (\setword{full existence}{FEX}) Given $A,B,C$ there exists $A'$ such that $A'\equiv_C A$ and $A\ind_C B$.
  \item (\setword{extension}{EXT}) If $A\ind_C B$ then for any $D$ there is $A'\equiv_{BC} A$ with $A'\ind_C BD$.
  \item (\setword{symmetry}{SYM}) If $A\ind_C B$ then $B\ind_C A$.
  \item (\setword{transitivity}{TRA}) Given $C\seq D\seq A$, if $A\ind_{D} B$ and $D\ind_C B$ then $A\ind_C B$.
  \item (\setword{local character}{LOC}) For every $A$ and $B$ there exists $C\subseteq B$ such that $\abs{C}\leq \abs{A}+\abs{T}$ and $A\ind_{C} B$.
   \item (\setword{chain local character}{LOCS}) Let $a$ be a finite tuple and $\kappa >\abs{T}$ a regular cardinal. For every continuous chain $(M_i)_{i<\kappa}$ of models with $\abs{M_i}<\kappa$ for all $i<\kappa$ and $M=\bigcup_{i<\kappa} M_i$, there is $j<\kappa$ such that $a\ind_{M_{j}} M$.
  \item (\setword{the independence theorem}{INDTHM} over models) Let $M$ be a small model, and assume $A\ind_M B$, $C_1\ind_M A$, $C_2\ind_M B$, and $C_1\equiv_M C_2$. Then there is a set $C$ such that $C\ind_M AB$, $C\equiv_{MA}C_1$, and $C\equiv_{MB}C_2$.
  \item (\setword{stationarity}{STAT}) Assume $C_1\ind_B A$, $C_2\ind_B A$, and $C_1\equiv_B C_2$. Then $C_1\equiv_{AB} C_2$.
\end{enumerate}
\end{definition}




A well-known result of Kim and Pillay \cite{KP97} gives a characterisation of simple theories and forking by the existence of an invariant ternary relation satisfying a certain set of axioms. Chernikov and Ramsey \cite{CR16} and Kaplan and Ramsey \cite{KR20} provide a similar characterisation of NSOP$_1$ theories and the so-called Kim-forking. In \cite{DK22} Dobrowolski and Kamsma extended those result to the positive setting and, in doing so, yield a refined version of the result of Chernikov, Ramsey and Kaplan, when translated back into the first-order setting. This is the version of the Kim-Pillay style characterization of NSOP$_1$ theories and Kim-independence that we state now, as in \cite[Fact 4.6]{CdEHJRK22}.

\begin{fact}[Chernikov-Ramsey; Kaplan-Ramsey; Dobrowolski-Kamsma]\label{fact:KPnsop1} 
A complete theory $T$ is $\NSOP 1$ if and only if there is an  invariant ternary relation $\ind$ on small subsets of $\MM$, which satisfies \ref{SYM} over models, \ref{EX} over models, \ref{FIN} over models, \ref{MON} over models, \ref{TRA} over models, \ref{EXT} over models,  \ref{INDTHM} over models,  and \ref{LOCS}.
Moreover, in this case $\ind$ is Kim-independence over models.
\end{fact}

\subsection{Any completion of ACFH is NSOP$_1$} We work in a monster model $(\K,\theta)$ of a completion of ACFH. 

\begin{definition}
For small sets $A,B,C$ in $\K$, we define
\[A\indi\theta _C B \iff \cl_\theta(AC)\indi{\alg}_{\cl_\theta(C)} \cl_\theta(BC).\]
\end{definition} 

We need to check that $\indi{\theta}$ satisfies all properties in Fact \ref{fact:KPnsop1}. The properties \ref{SYM}, \ref{EX}, \ref{FIN}, \ref{MON} and \ref{TRA} immediately follow from the fact that they are satisfied by $\indi{\alg}\ \ $. The property \ref{EXT} follows from \ref{TRA} and \ref{FEX}. It remains to check \ref{FEX} and \ref{LOCS}. 
\begin{proof}[Proof of \ref{FEX}]
Let $A,B,C$ be small subsets of $\K$. We may assume that $A,B,C$ are $\cl_\theta$-closed and that $C\seq A\cap B$. Let $D = \cl_\theta(AB)$. By Remark \ref{rk_isomorphiccopy}, there exists $(D',\theta')\models T$ and an $\LL$-isomorphism $\sigma:(D',\theta')\cong (D,\theta\upharpoonright D)$ over $C$ such that $D'\indi \alg _C D$. By saturation of $(\K,\theta)$, we may assume that $D'\seq \K$. By considering $A' = \sigma^{-1}(A)$, we have $A'\equiv_C A$ by Proposition \ref{prop:typesACFH} and $A'\indi \alg _C B$ by \ref{MON} of $\indi \alg\ \ $.
\end{proof} 

Instead of \ref{LOCS}, we prove that $\indi\theta$ satisfies the stronger property \ref{LOC}.

\begin{proof}[Proof of \ref{LOC}]Let $A$ and $B$ be given. We may assume that $B = \oT B$. By \ref{LOC} for ACF, there exists $C_0\seq B$ such that $\abs{C_0}\leq \abs{A}+\abs{T}$ and such that $A\indi{\alg}_{C_0} B$. Again by \ref{LOC} for $\indi{\alg}\ $, there exists $C_1\seq B$ such that $\abs{C_1}\leq \abs{\oT{AC_0}}+\abs{T} = \abs{A}+\abs{T}$ such that $\oT{AC_0} \indi{\alg}_{C_1} B$. By \ref{BMON} for $\indi\alg\ \  $, we may assume that $C_0\seq C_1\seq B$. Inductively, we find a chain $(C_n)_{n<\omega}$ such that for all $n$, $\oT{aC_{n}} \indi{\alg}_{C_{n+1}} B$, with $C_{n}\seq C_{n+1}\seq B$ and $\abs{C_n}\leq \abs{A}+\abs{T}$. Let $C = \bigcup_{n<\omega} C_n$. Then, by construction $\oT{AC_n }\indi\alg_C B$ for all $n$, so by \ref{FIN} for $\indi{\alg}\ \ $, we have $\oT{AC}\indi{\alg}_C B$, hence $A\indi{\theta}_C B$. 
\end{proof}

It is standard that \ref{INDTHM} follows from $3$-amalgamation
(Theorem \ref{thm_n-AM}), see \cite[Proposition A.4]{dEKN21} for instance. We conclude.

\begin{theorem}\label{thm_ACFH_NSOP1}
Any completion of ACFH is NSOP$_1$ and $\indi\theta$ coincides with Kim-independence over models.
\end{theorem}

\begin{remark}\label{rk_independencetheoremoveracl}
In ACFH, we actually get a strong version of \ref{INDTHM}, analogously to ACFA or ACFG. First, \ref{INDTHM} holds over every $\cl_\theta$-closed sets, not only over models. For ACFA in \cite[Generalised independence theorem, (1.9)]{CH99}, $E$ is an algebraically closed substructure. For ACFG, this is \cite[Example 5.2]{dE21A}. Further, the parameters $A$ and $B$ need not be independent, it is enough that they intersect in the base set. This is folklore for ACFA and is explicit in \cite[Example 5.2]{dE21A} for ACFG. 
\end{remark}

\subsection{Failure of \ref{BMON} and TP$_2$.}

We exemplify the failure of \ref{BMON} for $\indi \theta$. Let $(\K,\theta)$ be a monster model of ACFH. Let $E = \oT{E}$, $a,b,c$  in $\K$ such that $a,b,c$ are algebraically independent over $E$. Assume that $a,b,c$ are solutions of the equation $\theta(x+y) = z$ which further satisfy:
\begin{enumerate}
    \item $\theta[E(a)^\alg]\seq E$,
    \item $\theta[E(b,c)^\alg]\seq E$,
    \item $\theta(a+b)=c$,
    \item $\theta[E(a,b)^\alg] = \vect{E, c}$.
\end{enumerate}
Then $\oT{E(a)} = E(a)^\alg$, $\oT{E(b,c)} = E(b,c)^\alg$,  $\oT{E(b)} = E(b)^\alg$, hence $a\indi\theta_E bc$. However, $c\in \oT{E(a,b)}$, hence $c\in (\oT{E(a,b)}\cap \oT{E(b,c)})\setminus \oT{E(b)}$ so in particular $a\nindi\theta_{Eb} c$.

We explain how to formally get (1) above, (2), (3) and (4) are similar. Assume that $E$ is a subset of $(\K,\theta)$ such that $E = \oT E$, and let $\theta_E = \theta\upharpoonright E$. Let $x$ be algebraically independent over $\K$, in particular over $E$. As $E^\times$ is divisible, there exists $H$ such that $E\odot H = (Ex)^\alg$. Then define $\theta_H:H\rightarrow \set 1$ and $\theta' := \theta_E\odot\theta_H : (Ex)^\alg\rightarrow E$. In particular $((Ex)^\alg,\theta')$ and $(\K,\theta)$ are independent extensions of $(E,\theta_E)$ hence by $2$-amalgamation of $T$ and model completeness of ACFH, the type associated to the isomorphism type of $((Ex)^\alg,\theta')$ is consistent in $(\K,\theta)$, hence we may find $a\in \K$ such that $((Ex)^\alg,\theta')\cong^\LL_E ((Ea)^\alg,\theta)$.

TP$_2$ in ACFH can be witnessed by the presence of the kernels, which are generic subgroups, by the formula $xy+z\in G$, as in \cite[Theorem 4.2.]{dEKN21}. TP$_2$ in ACFH also follows from the presence of the generic non-linear binary map defined by $(x,y)\mapsto \theta(x+y)$, as in \cite[Proposition 3.14.]{KrR18}.

\subsection{Elimination of imaginaries under the existence axiom}
 In this subsection we prove that ACFH has elimination of imaginaries, under the condition that forking satisfies \ref{EX}.
We use the following standard definitions.
 \begin{enumerate}
 \item[$\text{ }$] $a\indi a _C b$ if and only if $acl(Ca)\cap acl(Cb) = acl(C)$
  \item[$\text{ }$] $a\indi{K}_C b$ if and only if $tp(a/Cb)$ does not Kim-fork over $C$
 \item[$\text{ }$] $a\indi{d}_C b$ if and only if $tp(a/Cb)$ does not divide over $C$ 
 \item[$\text{ }$] $a\indi{f}_C b$ if and only if $tp(a/Cb)$ does not fork over $C$
 \end{enumerate}

 We quickly introduce some notations from what could be called `axiomatic independence theory', which was developed by Adler in his thesis \cite{adlerthesis} (see also \cite{A09}), then further used in various papers in the recent study of NSOP$_1$ theories (e.g. \cite{CK19,dE21B,KR20,KrR18}), but is rooted in the study of forking in simple theories \cite{KP97}. See \cite{dE23axiomaticindep} for a recent treatement of this topic.

\begin{definition}
Let $\ind$ be any ternary relation, we define $\indi{m} $ and $\indi *$ as follows.
\begin{itemize}
    \item (Forcing base monotonicity) $A\indi{m}_C  B$ if $A\ind_{CD} BC$ for all $D\subseteq \acl(BC)$.
    \item (Forcing extension) $A\indi{*}_C B$ if $\forall \hat{B}\supseteq B$ there exists $A'\equiv_{BC}A$ such that $A'\ind_C \hat B.$
\end{itemize}
\end{definition}

\begin{example}
In ACF, the relation ${\indi a}^m$ is $\indi \alg$. The relation $\indi M\ $ in~\cite[Section 4]{A09} is the relation ${\indi a}^{m}$ in our context. 
\end{example}
The following is \cite[Lemma 3.1]{A09} and \cite[Lemma 3.2]{dE21B}.

\begin{fact}\label{fact_mon_ext}
  Let $\ind $ be an invariant ternary relation satisfying \ref{MON}, \ref{TRA}. 
  \begin{itemize}
      \item The relation $\indi{m} $ is invariant and satisfies \ref{MON}, \ref{TRA} and \ref{BMON}.
      \item The relation $\indi{*}$ is invariant and satisfies \ref{MON}, \ref{TRA} and \ref{EXT}. If $\ind$ satisfies \ref{BMON}, so does $\indi *$.
  \end{itemize}
\end{fact}

  Let $\ind$, $\indi 0$ be two ternary relations, such that $\ind \rightarrow \indi 0$, by which we mean $A\ind_C B$ implies $A\indi 0_C B$, for all $A,B,C$. We also say that $\ind$ is \textit{stronger} than $\indi 0$. If $\ind$ satisfies \ref{BMON} then $\ind$ is stronger than ${\indi 0}^{m}$.


The following follows from \cite[Lemma 4.12]{dE21B} (see also \cite[Theorem 4.1.24]{dE23axiomaticindep} for a more general version). 

\begin{fact}\label{fact_fork}
  Let $\ind$ be a ternary invariant relation, which satisfies  \ref{MON} and \ref{INDTHM} over algebraically closed sets. If for all $a,b,C$ $\tp(a/bC)$ if finitely satisfiable in $C$ implies $a\ind_C b$, then ${\indi{m}\ } ^{*} \rightarrow \indi f$. 
\end{fact}

\begin{remark}
It is actually in the current folklore that $\indi f = {(\indi K\ )^m}^*$ over models in an \emph{arbitrary} theory. For a proof, ask Kaplan or Ramsey.
\end{remark}

\begin{remark}
In particular, in ACFH, $(\indi {\theta m}\ \ )^{*} = \indi f$. As for now, we do not know if $\indi f$ satisfies \ref{EX}. We will take it as an assumption in Proposition \ref{prop:imextension}.
\end{remark}

  The following classical fact follows from a group theoretic lemma due to P.M. Neumann~\cite{Neu76}. To our knowledge, it appears first in~\cite[Lemma 1.4]{EvaHru93}.
 \begin{fact}\label{neumann}
  Let $\MM$ be a highly saturated model, $X$ a $0$-definable set, $e\in \MM$, $E = \acl(e)\cap X$ and a tuple $a$ from $X$. Then there is a tuple $b$ from $X$ such that 
  $$a\equiv_{Ee} b \mbox{ and } \acl(Ea)\cap \acl(Eb)\cap X = E.$$
\end{fact}

\begin{proposition}\label{prop:imextension}
Assume that $\indi f$ satisfies \ref{EX}. Let $a,b\in \K$, $e\in \dcl^\eq(a)$ and $E = \acl^\eq(e)\cap \K$. Then there exists $a'\equiv_{Ee} a$ such that $a'\indi{\theta}_E b$.
\end{proposition}
\begin{proof}
By \ref{MON} of $\indi \theta$, it is enough to prove that if $a\in \K$ and $e\in \dcl^\eq(a)$, $E = \acl^\eq(e)\cap \K$, there exists $a'\equiv_{Ee} a$ such that $a'\indi \theta_E a$. 
Let $a,E,e$ be as in the hypotheses. By Fact \ref{neumann}, there exists $b\equiv_{Ee} a$ such that 
\[\acl^\eq(Ea)\cap \acl^\eq (Eb)\cap \K  = \acl_\theta(Ea)\cap \acl_\theta(Eb)= E.\] 
We assume that $a$ and $b$ are enumerations of $\acl_\theta(Ea)$ and $\acl_\theta(Eb)$ such that $a\equiv_{Ee} b$, so $a\cap b = E$. 
We construct a sequence $(a_i)_{i<\omega}$ such that $$a_{n+1} \indi f _{a_n} a_n\ldots a_{0} \mbox{ and }a_n a_{n+1}\equiv_{E} ab.$$

 Start by $a_0 = a$ and $a_1 = b$. Assume that $a_0,\dots,a_n$ have already been constructed. We have that $a_{n-1} \equiv_{E} a_n$ so let $\sigma$ be an $E$-automorphism of $\K$ such that $\sigma(a_{n-1}) = a_n.$
By \ref{EX} for $\indi f$ there exists $a_{n+1} \equiv_{a_n } \sigma(a_n)$ such that $a_{n+1}\indi f _{a_n} a_n\dots a_{0}$. It follows that 
\[a_n a_{n+1} \equiv_{E} a_n\sigma(a_n)\equiv_{E} a_{n-1}a_n.\]
Let $(a_i)_{i<\omega}$ be such a sequence. Note that as $e\in \dcl^\eq(a)$, $a\equiv_{Ee} b$ and $a_na_{n+1}\equiv_E ab$, we have $a_n\equiv_{Ee} a$ for all $n<\omega$. 

Let $0<i<n$, then by \ref{NOR}, \ref{MON}, \ref{BMON} and \ref{TRA} we obtain
\[a_n\ldots a_{i+1}\indi f_{a_i} a_i\ldots a_0.\]
Using the previous expression for $i+1$, we may apply \ref{NOR} and \ref{MON}, to get
\[a_n\ldots a_{i+1}\indi f_{a_{i+1}} a_i\ldots a_0.\]
Hence $\acl_\theta(a_n\ldots a_{i+1})\cap \acl_\theta(a_i\ldots a_0)\seq a_i\cap a_{i+1} = E$ as $a_ia_{i+1}\equiv_E ab$. Then, for all $j_1>\ldots>j_s>i_1>\ldots>i_k$ we have 
\[
\acl_\theta(a_{i_1}\ldots a_{i_k})\cap \acl_\theta(a_{j_1}\ldots a_{j_s}) = a_{i_1}\cap a_{i_1+1} =  E.\quad \quad (\star)
\]
 We construct an indiscernible sequence associated to $(a_i)_{i<\omega}$. Let $\U$ be a non-principal ultrafilter on $\omega$ and let $p$ be the average type of $(a_i)_{i<\omega}$ over $\K$, i.e. $\phi(x)\in p$ if and only if $\set{i<\omega\mid \K\models \phi(a_i)}\in \U$ for all $\phi(x)$ with parameters in $\K$. Let $A = \set{a_i\mid i<\omega}$ and define the sequence $(c_i)_{i<\omega}$ by $c_0 \models p\upharpoonright A$ and $c_{i+1}\models p\upharpoonright Ac_0\ldots c_{i}$. 
The essential property of the sequence $(c_i)_{i<\omega}$ is that for all formulas $\phi(x_1,\ldots,x_s)$ over $A$, for all $j_1>j_2>\ldots >j_s$ we have $\phi(c_{j_1},\ldots,c_{j_s})$ if and only if there is $I_1\in \U$ such that for all $a_{k_1}\in I_1$, there exists $I_2 = I_2(a_{k_1})\in \U$ such that for all $a_{k_2}\in I_2$ there is $I_3$ (depending on $a_{k_1},a_{k_2}$), etc and there exists $I_s$ (depending on the previous choices) such that for all $a_{k_s}\in I_s$ we have $\phi(a_{k_1},\ldots,a_{k_s})$. In particular, $(c_i)_{i<\omega}$ is an indiscernible sequence over $A$ and satisfies the EM-type of $(a_i)_{i<\omega}$ over $E$, and we have $c_n\equiv_{Ee} a$ for all $n<\omega$. Further, for any $j_1>\ldots>j_s>i_1>\ldots >i_t$, the type $\tp(c_{j_1}\ldots c_{j_s}/Ac_{i_1}\ldots c_{i_t})$ is finitely satisfiable in $A$ by tuples $a_{k_1},\ldots,a_{k_s}$ of arbitrarily large indices $k_1>k_2>\ldots >k_s$. 
\begin{claim}
    $\acl_\theta(c_{j_1}\ldots c_{j_s})\cap \acl_\theta(c_{i_1}\ldots c_{i_t}) = E$, for all $j_1>\ldots>j_s>i_1>\ldots >i_t$.
\end{claim}
\begin{proof}[Proof of the claim]
    First, we prove that $\acl_\theta(c_{j_1}\ldots c_{j_s})\cap \acl_{\theta}(A)= E$. By contradiction, assume that some element $d$ of $\acl_\theta(c_{j_1}\ldots c_{j_s})\cap \acl_{\theta}(A)$ is not in $E$. Let $\phi(x,\vec c)$ be an algebraic formula over $E$ witnessing $d\in \acl_\theta(c_{j_1}\ldots c_{j_s})$ with $\vec c = c_{j_1}\ldots c_{j_s}$ and let $\psi(x,a_{i_1},\ldots,a_{i_t})$ be an algebraic formula over $E$ witnessing $d\in \acl_\theta(A)$ which has no realisations in $E$. As $\exists x( \phi(x,\vec c)\wedge \psi(x,a_{i_1},\ldots,a_{i_t}))\in \tp(\vec c/A)$, there are $k_1>\ldots>k_s>i_1>\ldots >i_t$ such that $\exists x \phi(x,a_{k_1},\ldots,a_{k_s})\wedge \psi(x,a_{i_1},\ldots,a_{i_t})$ is consistent, which contradicts $(\star)$. A very similar argument using $\acl_\theta((c_i)_{i<\omega})\cap \acl_{\theta}(A)= E$ and the fact that $\tp(c_{j_1}\ldots c_{j_s}/Ac_{i_1}\ldots c_{i_k})$ is finitely satisfiable in $A$ yields $\acl_\theta(c_{j_1}\ldots c_{j_s})\cap \acl_\theta(c_{i_1}\ldots c_{i_t})\seq \acl(A)$ and hence the claim.
\end{proof}

The sequence $(c_i)_{i<\omega}$ is indiscernible over $E$ and satisfies the claim. By stability and elimination of imaginaries of the reduct ACF (see \cite[Remark 4.8]{Hru12}) we have $c_{n+1}\indi \alg _{c_{n}} c_{n-1}\ldots c_0$ for all $n<\omega$.

 We prove that $c_2\indi{\theta } _E c_0$. 
Note that $(c_i)_{1\leq i<\omega}$ is indiscernible over $c_0$ in ACFH. In particular, this sequence is indiscernible in the sense of the stable reduct ACF, hence $(c_i)_{1\leq i<\omega}$ is totally indiscernible over $c_0$ in the sense of ACF hence \[c_2c_1 \equiv^{\ACF}_{c_0} c_1c_2.\quad \quad (\star\star)\]
As $c_2 \indi{\alg}_{c_1} c_0$, we get $c_1 \indi{\alg}_{c_2} c_0$ by $(\star\star)$.
By elimination of imaginaries in ACF, the canonical base of $\tp^\ACF(c_0/c_1c_2)$ is contained in $c_1$ and $c_2$, hence in $E = c_1\cap c_2$ so that in particular $c_2\indi{\alg}_{E}\  c_0$. As $c_0$ and $c_2$ are $\acl_\theta$-closed, we get $c_2\indi \theta _{E} c_0$. Let $\tau$ be an automorphism over $Ee$ sending $c_0$ to $a$, we conclude by taking $a' = \tau(c_2)$. \end{proof}

\begin{theorem}\label{thm_wei}
If $\indi f$ satisfies \ref{EX}, then ACFH has elimination of imaginaries.
\end{theorem}

\begin{proof}
We first prove that ACFH has weak elimination of imaginaries. Let $e\in \K^\eq$, there is a tuple $a$ from $\K$ and a $0$-definable function $f$ such that $f(a) = e$. Let $E = \acl^\eq(e)\cap \K$, we prove that $e\in \dcl^\eq(E)$. By contradiction, assume not, hence there exists $e'\neq e$ with $e\equiv_E e'$. By Proposition \ref{prop:imextension}, there exists $a'\equiv_{Ee} a$ such that $a'\indi \theta _E a$. Then there exists $b_1,b_2$ such that 
\[eaa'\equiv_E e'b_1b_2,\]
so that $f(b_1) = f(b_2) = e'$ and $b_1\indi\theta_E b_2$. By extension, there exists $b\equiv_{Eb_1} b_2$ such that $b\indi\theta _E a'$. By the independence theorem over algebraically closed sets (Theorem \ref{thm_n-AM}, see also Remark \ref{rk_independencetheoremoveracl}), there exists $c$ such that $c\equiv_{Ea} a'$ and $c \equiv_{Ea'} b$. From $c\equiv_{Ea} a'$ we have $f(c) = f(a) = e$ and from $c \equiv_{Ea'} b$ we have $f(c) \neq f(a') = e$, a contradiction. We conclude that ACFH has weak elimination of imaginaries. As any theory of fields eliminates finite imaginaries, we conclude.
\end{proof}

\begin{remark}
    The previous proof of weak elimination of imaginaries is rooted in classical arguments for elimination of imaginaries that appear for instance in~\cite{CP98}, \cite{CH99} or \cite{Hru12}. Various criteria for weak elimination of imaginaries have been developed by mimicking these arguments, see \cite[Proposition 4.25]{CK19}, \cite[Proposition 1.17]{MRK21}, \cite[Lemma 2.12]{dE21B}.
\end{remark}

It is a current open question whether forking and dividing coincide for types in NSOP$_1$ theories. Recall that in ACFH, we have $\indi f = (\indi{\theta m}\ \ )^*$ hence as $\indi f = \indi d ^*$ it is natural to ask the following question:
\begin{question}
    In ACFH, is $\indi f$ equal to $\indi{\theta m}\ \ $?
\end{question}

\section{Models of ACFH}\label{section_modelsofACFH}

The characterisation of existentially closed models of $T$ given in Theorem \ref{thm_genhom_mult} and the reduction to affine curves in Subsection \ref{subsec:reductiontoaffinecurves} are convenient for proving that a given structure is an existentially closed model, but not so much for describing the structure of a given existentially closed model of $T$. For this section it will be more convenient to deal with definable sets rather than affine varieties, so we give the following characterisation of models of $\ACFH$.

\begin{theorem}\label{thm_genhom_mult_version_defsets}
  $(K,\theta)\models \ACFH$
  if and only if for all complete systems of minimal equations $\tau(x;y)$ over $K$ with $\abs{x} = r$ and $\abs{y} = t$ and for all non-empty multiplicatively $r$-free constructible sets $X\seq \tau(K)\times \tau^\theta(K)\seq K^{r}\times K^t\times K^r\times K^t$ there exists tuples $a\in K^r,b\in K^t$ such that $(ab,\theta(ab))\in X$. 
\end{theorem}
\begin{proof}
This characterisation implies the characterisation in Theorem \ref{thm_genhom_mult_version_axioms}, hence it is enough to show that every model of ACFH satisfies the right hand side. We reprove Claim \ref{claim_extension_sat_tau} in Theorem \ref{thm_genhom_mult} replacing $V$ by $X$. Consider a complete system of minimal equations $\tau(x,y)$ over $K$ and a multiplicatively $r$-free definable set $X\subset \tau\times \tau^\theta$. Then by Lemma \ref{lm_char_r-free}, there is an extension $L\succ K$ and $(ab,a'b')\in X(L)$ such that the tuple $a$ is multiplicatively independent over $K$. As $L\models \tau(a,b)\wedge \tau^\theta(a',b')$, there is an endomorphism $\theta'$ of $L^\times$ such that $\theta(ab) = a'b'$ (Lemma \ref{lm_ext_mult}). The result follows by existential closedness.
\end{proof}

Note that it follows from Fact~\ref{fact_multi_free_r_firstorder} that the above characterisation of existentially closed models of $T$ is also first-order.

\begin{example}
Let $(K,\theta)$ be a model of ACFH. Then, for any definable set $X\subseteq K^{r+r}$ which is multiplicatively $r$-free, there exists $a\in K^r$ such that $(a,\theta(a))\in X$. Indeed, consider a trivial complete system of minimal equations $\tau\seq K^r\times K^t$ over $C$, for some natural numbers $r,t$. It consists of equations of the form $y_i = cx_1^{l_1}\ldots x_r^{l_r}$ (because of the conditions on the $\gcd$). From $X\subseteq K^{r+r}$, we can define $Y\seq K^r\times K^t\times K^r\times K^t$ such that $Y \seq \tau\times \tau^\theta$ (simply add to formula defining $X$ the new variables $y_i$ and the equations $y_i = c x_1^{l_1}\ldots x_r^{l_r}$) and for any $(a,b,c,d)\in Y$ we have $(a,c)\in X$, and $b,d$ are entirely determined by $a$ and $c$). In particular $Y$ is also multiplicatively $r$-free. It follows from Theorem \ref{thm_genhom_mult_version_defsets} that there exists $a\in K^r$ such that $(a,\theta(a))\in X$. 
\end{example}

\subsection{A ring of definable endomorphisms of a model of ACFH}
Let $(K,\theta)\models \ACFH$. For $n\in \N$, we denote $\theta^{(n)}= \theta\circ\dots\circ\theta$ the $n$-th iterate of $\theta$, and $\theta^{(0)} = \Id$.

We denote $\Endd(K^\times)$ the ring of definable endomorphisms of $K^\times$ in the language $\LL$, it is a ring where the `addition' is given by $(\phi_1\cdot \phi_2)(x) = \phi_1(x)\phi_2(x)$ and the `multiplication' is given by the composition of maps.

Let $P(X) = \sum_{i=0}^{n} k_iX^i\in \Z[X]$, we define $P(\theta)$ to be the endomorphism of $K^\times$ given by $x\mapsto x^{k_0}\theta(x^{k_1})\ldots \theta^{(n)}(x^{k_n})$. The corresponding map:
\begin{align*}
    \Phi:\Z[X] &\rightarrow \Endd(K^\times)\\
    P(X)&\mapsto P(\theta)
\end{align*} 
is a ring homomorphism. We denote by $\Z[\theta]$ the image of $\Phi$, i.e. the subring of $\Endd(K^\times)$ given by $\set{P(\theta)\mid P(X)\in \Z[X]}$.

\begin{observation}\label{obs_keyformulaiteration}
Let $n\in \N^{>1}$. The $\LLf$-formula \[\Sigma(x_1,\ldots,x_n,y_1,\ldots,y_n) = \bigwedge_{i=2}^{n} x_i = y_{i-1}\] is multiplicatively $n$-free and enjoys the following property: $(a,\theta(a))\in \Sigma$ if and only if $a_2 = \theta(a_1), a_3 = \theta^{(2)}(a_1),\ldots,a_n = \theta^{(n-1)}(a_1)$.
\end{observation}

\begin{theorem}\label{thm_genpoly}
Let $Y\seq K^n$ be a multiplicatively $n$-free constructible set and $P_1,\ldots,P_n\in \Z[X]\setminus\Z$. Let $\delta_1,\ldots,\delta_n\in K$. Then there exists a tuple $(a_1,\ldots,a_n)\in Y$ such that $P_1(\theta)(a_1) = \delta_1,\ldots,P_n(\theta)(a_n) = \delta_n$.
\end{theorem}

\begin{proof}
Assume that $Y\seq K^n$ is defined by the formula $\psi(z_1,\ldots,z_n)$. 
Let $P_1 = \sum_{j=0}^{d_1} k_{1,j} X^j,\ldots,P_n = \sum_{j=0}^{d_n} k_{n,j} X^j$ with $k_{1,d_1}\neq 0,\ldots,k_{n,d_n}\neq 0$ and $d_1>0, \ldots, d_n>0$. By Observation \ref{obs_keyformulaiteration}, for $d>0$, the formula $\Sigma_{d}(x_0,\ldots,x_d,y_0,\ldots,y_d) = \bigwedge_{i=1}^{d} x_i = y_{i-1}$ is such that $(\vec a,\theta(\vec a))\models \Sigma$ if and only if $a_i = \theta^{(i)}(a_0)$. Any $(\vec a,\theta(\vec a))$ satisfying 
\[\varphi_i(\vec x,\vec y):= \Sigma_{d_i}(\vec x,\vec y)\wedge x_0^{k_{i,0}}\ldots x_d^{k_{i,d}} = \delta_i\] satisfies $P_i(\theta)(a_0) = \delta_i$. 
Let $\set{z_{i,j}, t_{i,j}\mid 1\leq i\leq n; 0\leq j\leq d_i}$ be a new set of variables and consider the formula $\Gamma((z_{i,j}); (t_{i,j}))$ defined by the following:
\begin{align*}
    \psi(z_{1,0},\ldots,z_{n,0})&\wedge \varphi_1(z_{1,0},\ldots,z_{1,d_1}, t_{1,0},\ldots,t_{1,d_1})\\
    &\wedge \varphi_2(z_{2,0},\ldots,z_{2,d_2}, t_{2,0},\ldots,t_{2,d_2})\\
    &\vdots\\
    &\wedge \varphi_n(z_{n,0},\ldots,z_{n,d_n}, t_{n,0},\ldots,t_{n,d_n})
\end{align*}
We need to show that there exists a point $(\vec a,\theta(\vec a))\in \Gamma$.
First, we may assume that $\gcd(k_{1,0},\ldots,k_{1,d_1}) = 1$. Otherwise, change $P_1$ to $P'_1 = \sum_{j=0}^d \frac{k_{1,j}}{\gcd(k_{1,0},\ldots,k_{1,d_1})} X^j$, and $\delta_1$ to $\delta_1'$ such that $(\delta_1')^{\gcd(k_{1,0},\ldots,k_{1,d_1})} = \delta_1$. Similarly, for each $1\leq i\leq n$, we assume that $\gcd(k_{i,0},\ldots,k_{i,d_i}) = 1$.

Let $(\vec u,\vec v)$ be a generic of $\Gamma$. Then $\vec u = (u_{i,j})$ satisfies the following equations (later refereed to as  \textit{equations $(\alpha)$}):
\begin{align*}
    z_{1,d_1}^{k_{1,d_1}} &= \delta_1 z_{1,0}^{-k_{1,0}}\ldots z_{1,d_1-1}^{-k_{1,d_1-1}}\\ 
    z_{2,d_2}^{k_{2,d_2}} &= \delta_2 z_{2,0}^{-k_{2,0}}\ldots z_{2,d_2-1}^{-k_{2,d_2-1}}\\
    &\ldots\\
    z_{n,d_n}^{k_{n,d_n}} &= \delta_n z_{n,0}^{-k_{n,0}}\ldots z_{n,d_n-1}^{-k_{n,d_n-1}}
\end{align*}
\begin{claim}
    Those equations define a unique complete system of minimal equations associated to $N_1 = k_{d_1}>0,\ldots,N_n = k_{d_n}>0$ (hence an m-variety) that is satisfied by $\vec u$ for the partition $(u_{1,d_1},\ldots,u_{n,d_n} ; (u_{i,j})_{1\leq i\leq n ; 0\leq j\leq d_i-1})$. 
\end{claim}

\begin{proof}[Proof of the claim]
Let $l_1,\ldots l_n$ be such that $l_i\leq k_{i,d_i}$ and $\gcd(l_1,\ldots,l_n) = 1$. The set 
\[\set{x\in \N\mid \forall 1\leq i\leq n \ k_{i,d_i}\text{ divides } xl_i}\]
is nonempty (take $x = k_{1,d_1}\ldots k_{n,d_n}$), it admits a minimal element, we denote it $N$. We have $N l_i= k_{i,d_i}M_i$ for some $M_i\in \N$, and by minimality, $\gcd(N,M_1,\ldots,M_n) = 1$. We have that 
\begin{align*}
    (u_{1,d_1}^{l_1}\ldots u_{n,d_n}^{l_n})^{N} &= (u_{1,d_1}^{k_{1,d_1}})^{M_1}\ldots (u_{n,d_n}^{k_{n,d_n}})^{M_n}\\
    &=(\delta_1u_{1,0}^{-k_{1,0}}\ldots u_{1,d_1-1}^{-k_{1,d_1-1}})^{M_1}\ldots (\delta_n u_{n,0}^{-k_{n,0}}\ldots u_{n,d_n-1}^{-k_{n,d_n-1}})^{M_n}\\
    &= (\delta_1^{M_1}\ldots\delta_n^{M_n}) u_{1,0}^{-k_{1,0}M_1}\ldots u_{1,d_1-1}^{-k_{1,d_1-1}M_1}\dots u_{n,0}^{-k_{n,0}M_n}\ldots u_{n,d_n-1}^{-k_{n,d_n-1}M_n}
\end{align*}
Let \[R=\gcd(N,-k_{1,0}M_1,\ldots,-k_{1,d_1-1}M_1,\ldots,-k_{n,0}M_n,\ldots,-k_{n,d_n}M_n).\] We want to show that $R = 1$. By contradiction, assume that $q$ is a prime number dividing $R$. As $q$ divides $N$ and $Nl_1 = M_1k_{1,d_1}$, we have that $q$ divides $M_1$ or $k_{1,d_1}$. Assume that $q$ does not divide $M_1$, then on one hand, $q$ divides $k_{1,d_1}$. On the other hand $q$ divides $-k_{1,0}M_1,\ldots,-k_{1,d_1-1}M_1$, hence $q$ divides each $k_{1,j}$ for $0\leq j\leq d_1$, which contradicts $1 = \gcd(k_{1,0},\ldots,k_{1,d_1})$. We conclude that $q$ divides $M_1$, and similarly, $q$ divides $M_2,\ldots,M_n$, hence $q$ divides $\gcd(N,M_1,\ldots,M_n) = 1$, a contradiction. We conclude that $R = 1$, hence the above equation is an instance of a complete system of minimal equations.
\end{proof}

Let $\tau(z_{1,d_1},\ldots,z_{n,d_n} ; (z_{i,j})_{1\leq i\leq n ; 0\leq j\leq d_i-1})$ be the conjunction of equations $(\alpha)$ together with, for each $l_1,\ldots , l_n$ with $l_i\leq k_{i,d_i}$ and $\gcd(l_1,\ldots,l_n) = 1$, the equation:
\[(z_{1,d}^{l_1}\ldots z_{n,d}^{l_d})^{N} = (\delta_1^{M_1}\ldots\delta_n^{M_n}) z_{1,0}^{-k_{1,0}M_1}\ldots z_{1,d_1-1}^{-k_{1,d_1-1}M_1}\dots z_{n,0}^{-k_{n,0}M_n}\ldots z_{n,d_n-1}^{-k_{n,d_n-1}M_n}\]
for $N,M_1,\ldots,M_n$ as above (depending on $(l_1,\ldots,l_n)$). Then $\tau(K)$ is an m-variety. 

Let $\Gamma^0(\vec z,\vec t) = \Gamma(\vec z,\vec t)\wedge \tau(\vec z)\wedge \tau^\theta(\vec t)$. We now show that $\Gamma^0(K)$ projects m-generically onto $\tau(K)$, i.e. there is a generic $(\vec u,\vec v)$ of $\Gamma^0(K)$ such that $\vec u = (u_{1,d_1},\ldots,u_{n,d_n} ; (u_{i,j})_{1\leq i\leq n ; 0\leq j\leq d_i-1})$ is an m-generic of $\tau$, i.e. such that $(u_{i,j})_{1\leq i\leq n ; 0\leq j\leq d_i-1}$ is multiplicatively independent over $K$. 

We explicitely give such a generic, which also gives consistency of the formula $\Gamma^0$. As $Y$ is multiplicatively $n$-free, there exists a generic $(u_{1,0},\dots,u_{n,0})$ of $Y$ which is multiplicatively independent over $K$. We have $\psi(u_{1,0},\dots,u_{n,0})$. Let $\set{u_{i,j}\mid 1\leq i\leq n; 1\leq j\leq d_i-1}$ be a set of elements algebraically independent over $K(u_{1,0},\dots,u_{n,0})$. Let $(u_{1,d_1},\ldots,u_{n,d_n})$ be a tuple of solutions to the equations
\begin{align*}
    z_{1,d_1}^{k_{1,d_1}} &= \delta_1 u_{1,0}^{-k_{1,0}}\ldots u_{1,d_1-1}^{-k_{1,d_1-1}}\\ 
    z_{2,d_2}^{k_{2,d_2}} &= \delta_2 u_{2,0}^{-k_{2,0}}\ldots u_{2,d_2-1}^{-k_{2,d_2-1}}\\
    &\ldots\\
    z_{n,d_n}^{k_{n,d_n}} &= \delta_n u_{n,0}^{-k_{n,0}}\ldots u_{n,d_n-1}^{-k_{n,d_n-1}}
\end{align*}
For $\vec u = (u_{1,d_1},\ldots,u_{n,d_n} ; (u_{i,j})_{1\leq i\leq n ; 0\leq j\leq d_i-1})$, we have $\vec u \models \tau$. Now for $i = 1\ldots,n$ and $j = 0,\ldots,d_i-1$, let $v_{i,j} := u_{i,j+1}$. Then, for $i=1,\ldots,n$ and for all $x,y$, $\Sigma_{d_i}(x,u_{i,1},\ldots,u_{i,d_i},v_{i,0},\ldots,v_{i,d_i-1},y)$. Finally, let $(v_{1,d_1},\ldots,v_{n,d_n})$ be a tuple of solutions to the equations
\begin{align*}
    z_{1,d_1}^{k_{1,d_1}} &= \theta(\delta_1) v_{1,0}^{-k_{1,0}}\ldots v_{1,d_1-1}^{-k_{1,d_1-1}}\\ 
    z_{2,d_2}^{k_{2,d_2}} &= \theta(\delta_2) v_{2,0}^{-k_{2,0}}\ldots v_{2,d_2-1}^{-k_{2,d_2-1}}\\
    &\ldots\\
    z_{n,d_n}^{k_{n,d_n}} &= \theta(\delta_n) v_{n,0}^{-k_{n,0}}\ldots v_{n,d_n-1}^{-k_{n,d_n-1}}
\end{align*}
For $\vec v = (v_{1,d_1},\ldots,v_{n,d_n} ; (v_{i,j})_{1\leq i\leq n ; 0\leq j\leq d_i-1})$, we have $\tau^\theta(\vec v)$, and for $i=1,\ldots,n$ $\Sigma_{d_i}(u_{i,0},\ldots,u_{i,d_i},v_{i,0},\ldots,v_{i,d_i})$. Finally, $(\vec u,\vec v)$ is a generic of $\Gamma^0(K)$ by construction, and also by construction, $\vec u$ is an m-generic of $\tau(K)$.

As $\Gamma^0(K)\seq \tau(K)\times \tau^\theta(K)$ projects m-generically onto $\tau(K)$, we use the axioms (Theorem \ref{thm_genhom_mult_version_defsets}) to conclude that there exists $a = (a_1,\ldots,a_n)\in K^n$ such that $(a,\theta(a))\in \Gamma^0(K)$. By construction of $\Gamma^0$, we have $a\in Y$ and $P_1(\theta)(a_1) = \delta_1,\ldots,P_n(\theta)(a_n) = \delta_n$. 
\end{proof}

\begin{example}[$\ker P(\theta)+\ker Q(\theta) = K$]\label{ex_stablyembedded}
Let $P,Q\in \Z[X]\setminus\Z$, $b\in K$ and $Y\seq K^2$ be the variety defined by the equation $x+y = b$. By Theorem \ref{thm_genpoly}, there exists $(a_1,a_2)\in Y$ such that $P(\theta)(a_1)=1$ and $Q(\theta)(a_2)=1$, so $b\in \ker P(\theta)+\ker Q(\theta)$. In particular, for $P=Q$, we see that $\ker P(\theta)$ is stably embedded.
\end{example}

\begin{corollary}\label{cor_def_endo_surj}
Let $P\in \Z[X]\setminus\set{0}$, then $P(\theta)$ is surjective. 
\end{corollary}
\begin{proof}
If $P\in \Z\setminus\set{0}$, then $P(\theta) = \pw^n$ for some $n\neq 0$, so $P(\theta)$ is surjective since $K$ is algebraically closed. If $P\in \Z[X]\setminus \Z$, apply Theorem \ref{thm_genpoly} with $Y = K$.
\end{proof}

\begin{corollary}
The ring $\Z[\theta]$ is isomorphic to $\Z[X]$.
\end{corollary}
\begin{proof}
The map $\Phi:\Z[X]\rightarrow \Z[\theta]$ is a ring epimorphism. Let $P(X)\in \ker\Phi$, then $P(\theta)$ is not surjective, so by Corollary \ref{cor_def_endo_surj}, $P(X) = 0$. It follows that $\Phi$ is an isomorphism.
\end{proof}

\begin{question}
    Do we have $\Endd(K,\theta) = \Z[\theta]$?
\end{question}

\subsection{$\ker P(\theta)$ are generic multiplicative groups}

Let $\LL_G$ be the language $\LLr\cup\set{G}$ where $G$ is a unary predicate, and let $\ACF_G^\times$ be the $\LL_G$ theory of algebraically closed fields where $G$ is a distinguished multiplicative subgroup \textit{with $G = \vect{G}^\dv$}. From \cite[Theorem 5.27]{dE21A}, $\ACF_G^\times$ has a model-companion $\ACFG^\times$. We can deduce the following geometric characterisation of models of $\ACFG^\times$ from \cite[Theorem 1.5]{dE21A}.

\begin{fact}\label{fact_ACFG}
Let $(K,G)\models \ACF_G^\times$ be $\omega$-saturated. Then $(K,G)$ is existentially closed if and only if for every $n$ and $k\leq n$, for every non-empty multiplicatively $n$-free constructible set $Y\seq K^{n}$ over a finitely generated\footnote{Here we mean finitely generated in the sense of the pregeometry $\vect{\cdot}^\dv$, not in the sense of groups. Nontrivial divisible abelian groups are never finitely generated.} $C = \vect{C}^\dv$, there exists $(a_1,\dots,a_n)\in Y$ such that $\vect{a_1,\dots,a_k, G(C)}^\dv = G\cap \vect{a_1,\ldots,a_n,C}^\dv$.
\end{fact}

 A natural question to ask is whether $(K,\ker \theta)$ or $(K,\ker P(\theta))$ are models of $\ACFG^\times$ whenever $(K,\theta)\models \ACFH$. The answer is no, $(K,\ker \theta)$ and $(K,\ker P(\theta))$ are not necessarily models of $\ACF_G^\times$, because in general $\ker P(\theta) \neq \vect{\ker P(\theta)}^\dv$. 

\begin{remark}[$\ker\theta \neq \vect{\ker\theta}^\dv$]\label{rk_ker_notrootclosed}
Let $n\geq 1$ and $\zeta$ be a generator of $\mu_n$. As $\theta$ is surjective, there exists $b$ such that $\theta(b) = \zeta$. Let $a = b^n$, then $a\in \ker\theta$ but $b\notin \ker\theta$, since $\zeta\neq 1$.
In the particular case where $\mu_\infty\seq \ker\theta$, we actually get that $\ker\theta$ is not divisible. Consider the isomorphism $K^\times/\ker\theta \cong K^\times$. If $\ker\theta$ were divisible, then the quotient $K^\times/\ker\theta$ would be torsion-free (using $\mu_\infty\seq \ker\theta$), which contradicts $K^\times/\ker \theta \cong K^\times$. Note that the condition $\ker\theta\neq \vect{\ker\theta}^\dv$ does not imply that $\ker\theta$ is not divisible. Actually, there are completions of $\ACFH$ where $\ker\theta$ is divisible, see Example \ref{ex_ker_div}.
\end{remark}

However, $\ker \theta$ present some form of genericity, which we define now. 

\begin{definition}
Let $K$ be a field and $G$ a subgroup of $K^\times$. We say that $G$ is \textit{generic in $K$} if for every $n $ and $k\leq n$, for every non-empty multiplicatively $n$-free constructible set $Y\seq K^{n}$ over a finitely generated $C = \vect{C}$ there exists $(a_1,\dots,a_n)\in Y$ such that $\vect{a_1,\dots,a_k,G(C)} = G\cap \vect{a_1,\ldots,a_n, C}$.
\end{definition}

\begin{theorem}\label{thm_ker_generic}
Let $(K,\theta)\models \ACFH$ and $P\in \Z[X]\setminus\set{0}$. Then $\ker P(\theta)$ is generic in $K$. 
\end{theorem}
\begin{proof}
Let $Y\seq K^n$ be a multiplicatively $n$-free constructible set defined over a finitely generated set $C = \vect{C}$. Let $1\leq k \leq n$. The multiplicative group of any algebraically closed field is not finitely generated, hence there exist $\delta_i$ such that $\delta_i = 1$ for $1\leq i \leq k$ and $(\delta_j)_{k<j\leq n}$ is multiplicatively independent over the group $P(\theta)(C)$. By Theorem \ref{thm_genpoly}, there exists $(a_1,\ldots,a_n)\in X$ with $P(\theta)(a_i) = \delta_i$. We prove that $\vect{a_1,\ldots,a_n,C}\cap G = \vect{a_1,\ldots,a_k,G(C)}$ for $G = \ker P(\theta)$. First, for each $1\leq i\leq k$ we have $P(\theta)(a_i) = 1$ hence $a_i$ belongs to $G$ and so does $a_1^{s_1}\ldots a_k^{s_k}$ for each $(s_1,\ldots,s_k)\in \Z^k$. Let $s_1,\ldots,s_n\in \Z$.
If $a_1^{s_1}\ldots a_n^{s_n}c\in G$, then apply $P(\theta)$ to get $\delta_{k+1}^{s_{k+1}}\ldots \delta_n^{s_n} = P(\theta)(c)^{-1}$ hence by assumption $s_{k+1} = \ldots = s_n = 0$, so $a_1^{s_1}\ldots a_n^{s_n}c = a_1^{s_1}\ldots a_k^{s_k}c$. As $a_1^{s_1}\ldots a_k^{s_k}\in G$ we also have $c\in G$, hence $a_1^{s_1}\ldots a_n^{s_n}c\in \vect{a_1,\ldots,a_k,G(C)}$. The other inclusion is trivial.
\end{proof}

\begin{example}[Disjoint generic subgroups]
Let $P\in \Z[X]\setminus\Z$. Let $Q = P-1$. Then \[\fix P(\theta) := \set{x \mid P(\theta)(x) = x} = \ker Q(\theta)\]
Then both $\ker P(\theta)$ and $\fix P(\theta)$ are generic, and $\ker P(\theta) \cap \fix P(\theta) = \set{1}$.
\end{example}

\subsection{Iterations of generic is generic}\label{subsec_iteration}

We prove that iterating the generic endomorphism $\theta$ in a model of ACFH yields another generic endomorphism.


\begin{lemma}\label{lm_ext_iterate}
Let $(K,\theta)\models T$ and $(L,\zeta)\models T$ and extension of $(K,\theta^{(n)})$, for $n\in \N\setminus \set{0}$. Then there exists a field extension $F$ of $K$ and $L$ such that $(F,\theta)$ extends $(K,\theta)$ and $(F, \theta^{(n)})$ extends $(L,\zeta)$.
\end{lemma}

\begin{proof}
Let $L_0 = L$ and for each $1\leq i<n$, let $L_i\indi{\alg}_K\ L_0,\ldots,L_{i-1}$ with $L_i\cong_K L_0$. Let $\sigma_i:L_i\rightarrow L_{i+1}$ be a field isomorphism over $K$ witnessing $L_i\cong_K L_{i+1}$, for $i<n-1$. Each $\sigma_i$ defines a multiplicative homomorphism $L_i\rightarrow L_{i+1}$ fixing $K^\times $. As $K^\times$ is divisible, it is a direct factor in $L_0$, hence let $H_0$ be such that $L_0^\times = K \odot H_0$. Define inductively $H_i$ for $i = 1,\ldots,n-2$ by $H_{i+1} = \sigma_i(H_i)$. Then $L_i = K \odot H_i$, for each $0\leq i<n$. Each restriction $\theta_i := \sigma_i\upharpoonright H_i : H_i\rightarrow  H_{i+1}$ is a group isomorphism, and so is $\theta_0^{-1}\circ\ldots\circ \theta_{n-1}^{-1} : H_{n-1}\rightarrow H_0$. Note that $L_i\cap L_j = K$ hence $H_i\cap H_j = \set{1}$ for all $i\neq j$.

Let $F = (L_0\ldots L_{n-1})^\alg$.  Let $G = \vect{L_0,\ldots,L_{n-1}} = K^\times \odot H_0\odot \ldots \odot H_{n-1}$. $G$ is a subgroup of $F^\times$.

To define an endomorphism of $G$
we have the following homomorphisms on each factor:
\begin{itemize}
    \item on $K^\times$, we have $\theta: K^\times \rightarrow K^\times \seq G$; 
    \item on $H_i$, we have $\theta_i :H_i\rightarrow H_{i+1}\seq G$, for $i = 0,\ldots,n-2$;
    \item on $H_{n-1}$, we have $\theta_{n-1}:= \zeta\circ\theta_0^{-1}\circ\ldots\circ \theta_{n-2}^{-1}$, so $\theta_{n-1}: H_{n-1}\rightarrow K^\times \odot H_0 = L_0\seq G$.
\end{itemize}
We now define an endomorphism $\theta_G$ of $G$ by setting 
\[\theta_G(kh_0\ldots h_{n-1}) = \theta(k)\theta_0(h_0)\ldots \theta_{n-1}(h_{n-1}),\] for $k\in K,h_0\in H_0\ldots,h_{n-1}\in H_{n-1}$. As $F^\times$ is divisible, $\theta_G$ extends to an endomorphism $\theta_F$ of $F$. As $\theta_F\upharpoonright K = \theta$, $(F,\theta_F)$ extends $(K,\theta)$. Let $x\in L$, so $x = kh_0$ for $k\in K$ and $h_0\in H_0$. Then $\theta_F^{(n)}(x) = \theta_G^{(n)}(kh_0) = \theta^{(n)}(k)\theta_G^{(n)}(h_0)$. By definition, we have $\theta_G^{(n)} (h_0)= \theta_{n-1}(\theta_{n-2}\circ\ldots\circ \theta_0(h_0)) = \zeta(h_0)$. As $(L,\zeta)$ extends $(K,\theta^{(n)})$, we have $\theta^{(n)}(k) = \zeta(k)$, hence $\theta_F^{(n)}(x) = \zeta(k)\zeta(h_0) = \zeta(kh_0)$.
\end{proof}

\begin{proposition}\label{prop_iterations}
Let $(K,\theta)$ be a model of ACFH. Then, for all $n\in \N\setminus\set{0}$, $(K,\theta^{(n)})$ is a model of ACFH.
\end{proposition}
\begin{proof}
We prove that if $(K,\theta)$ is an existentially closed model of $T$, then so is $(K,\theta^{(n)})$. Let $\tilde \theta = \theta^{(n)}$, and let $(L,\zeta)$ be an extension of $(K,\tilde\theta)$. By Lemma \ref{lm_ext_iterate}, there exists $(F,\theta_F)$ extending $(K,\theta)$ such that $(F,\theta_F^{(n)})$ extends $(L,\zeta)$. Let $\varphi$ be an existential $\LL$-formula with parameters in $K$ such that $(L,\zeta)\models \varphi$. Then $(F,\theta_F^{(n)})\models \varphi$. Let $\tilde \varphi$ be the $\LL$-formula obtained by replacing each occurence of the function symbol $\theta$ in $\varphi$ by $\theta^{(n)}$. Then $\tilde \varphi$ is still an existential formula and $(F,\theta_F)\models \tilde \varphi$. As $(K,\theta)$ is existentially closed in $(F,\theta_F)$, we also have $(K,\theta)\models \tilde \varphi$ and hence $(K,\theta^{(n)})\models \varphi$.
\end{proof}

Of course, $(K,\theta)$ and $(K,\theta^{(n)})$ do not define the same completions of ACFH in general. For instance, in characteristic $0$, let $\theta$ be any multiplicative endomorphism of $\Q$ exchanging $2$ and $3$, such endomorphism exists since $2$ and $3$ are multiplicatively independent. Then $\theta(2) = 3$ but $\theta^{(2)}(2) = 2$, so they define different completions of ACFH. 


\begin{question}
    Let $(K,\theta)$ be a model of ACFH. Is $(K,P(\theta))$ also a model of ACFH, for all $P\in \Z[X]\setminus \Z$?
\end{question}

\begin{question}
For $(K,\theta)$ a model of $T$ or ACFH, is the ultraproduct $\prod_\U (K,P_n(\theta))$ a model of ACFH, for some sequence $(P_i)_{i<\omega}\in \Z[X]$, and $\U$ an ultrafilter on $\omega$?
\end{question}

\subsection{The kernels $\ker P(\theta)$ are pseudofinite-cyclic groups}\label{subsection_pseudofinite}

 A \textit{pseudofinite abelian group} is a group which is elementary equivalent to an ultraproduct of finite abelian groups. We momentarily switch to additive notation for abelian groups. Pseudofinite abelian groups have been classically studied by Basarab \cite{Bas75} (and more recently by Herzog and Rothmahler \cite{HR09}). They are are elementary equivalent to groups of the form
 \[\bigoplus_{p\text{ prime}} \left[ \oplus_{n>0}\Z(p^n)^{\kappa_{p,n}}\oplus (\Z(p^\infty)\oplus \Z_p)^{\lambda_p}\right]\oplus \Q^\epsilon\]
for $\kappa_{p,n}, \lambda_p$ finite or countable, and $\epsilon=0$ or $\omega$. 

We switch back to multiplicative notation for abelian groups.

\begin{definition}
A group is \textit{pseudofinite-cyclic} if it is elementarily equivalent to an ultraproduct of finite cyclic groups.
\end{definition}

The following criterion for pseudofinite-cyclic groups is joint with I. Herzog, the proof is given in Appendix \ref{appx_psfc}.

\begin{fact}[d'Elbée-Herzog]\label{fact_herzog}
An abelian group $G$ is pseudofinite-cyclic if and only if for all prime $p$ we have:
\[\abs{\mu_p(G)} = \abs{G/\pw^p(G)}\leq p.\]
\end{fact}

\begin{remark}
This is satisfied by any finite cyclic group because of the short exact sequence
\[1\rightarrow \mu_p(G)\rightarrow G \rightarrow \pw^p(G)\rightarrow 1.\]
Of course, in a finite cyclic group $C$ of finite order, $\mu_p(C)$ is also cyclic and of order $\leq p$ for all prime $p$.
\end{remark}

\begin{lemma}\label{lm_herzog_inequ1}
Let $G\seq D$ be abelian groups with $D$ divisible. Let $n\in \N$ and assume that $\abs{\mu_n(D/G)}\leq \abs{\mu_n(D)} <\infty$.
Then $\abs{G/\pw^n(G)}\leq \abs{\mu_n(G)}$.
\end{lemma}
\begin{proof}
Let $N = \abs{\mu_n(D)}$ and $M = \abs{\mu_n(G)}$. As $\mu_n(G)$ is a subgroup of $\mu_n(D)$, we have that $M$ divides $N$. Let $k = N/M$. By contradiction, assume that $\abs{G/\pw^n(G)}>M$, and let $g_1,\ldots,g_{M+1}\in G$ be such that $g_ig_j^{-1}\notin \pw^n(G)$, for $i\neq j$. 
As $D$ is $n$-divisible, for each $i$, there is $u_i\in D$ such that $\pw^n(u_i) = g_i$. Note that $u_iu_j^{-1}\notin G$, for all $i\neq j$. Let $k\geq 1$ and $\zeta_1,\ldots,\zeta_k$ be representatives of the classes of $\mu_n(D)$ modulo $\mu_n(G)$. For all $i,j$, we have $\pw^n(u_i\zeta_j)
= g_i$. 
\begin{claim}
    We have: $u_i\zeta_{i'} G= u_j\zeta_{j'}G$ if and only if $i = j$ and $i' = j'$.
\end{claim}
\begin{proof}[Proof of the claim]
Assume that $u_i\zeta_{i'} G= u_j\zeta_{j'}G$, then $u_iu_j^{-1} = \zeta_{j'}\zeta_{i'}^{-1} G$. Applying $\pw^{n}$, we get $g_ig_j^{-1} \in \pw^{n}(G)$, so $i=j$. If follows that $\zeta_{i'} G= \zeta_{j'}G$, so $i' = j'$. The converse is clear.
\end{proof}

From the claim that the set $\set{u_j\zeta_iG \mid 1\leq i\leq k, 1\leq j\leq M+1}$ is a subset of $\mu_n(D/G)$ of size $k(M+1) = N+k>N$. By hypothesis, $\abs{\mu_n(D/G)}\leq  N$, hence we reach a contradiction.
\end{proof}

The previous lemma generalizes the following: if $D$ is a divisible and torsion-free abelian group, $G$ a subgroup of $D$ such that $D/G$ is torsion-free, then $G$ is divisible.

\begin{example}[A completion of ACFH where $\ker\theta$ is a $\Q$-vector space]\label{ex_ker_div}
Let $F$ be a field containing all roots of $1$. Let $\theta_0: F^\times \rightarrow F^\times$ be a multiplicative monomorphism and $(K,\theta)$ a model of ACFH extending $(F,\theta_0)$. In particular, $K/\ker\theta\cong K$. Then $G:=\ker\theta$ is a $\Q$-vector space. Indeed, as $\ker\theta_0 = \set{1}$, we have $\mu_n(G) = \mu_n(\ker\theta_0) = \set{1}$ for all $n\in \N$, so $G$ is torsion-free. By Lemma \ref{lm_herzog_inequ1} with $D = K^\times$, we have $\abs{G /\pw^n(G)} = 1$, hence $\pw^n(G) = G$, i.e. $G$ is $n$-divisible. As it is torsion-free and divisible, it is a $\Q$-vector space. Note that in any model of this completion, $\ker\theta$ is pseudofinite: it satisfies Fact \ref{fact_herzog}: $\abs{\mu_n(G)} = \abs{G/\pw^n(G)} = 1\leq n$. But it also follows from the fact that ``$\ker\theta$ is divisible and torsion-free" is a first-order condition, and that torsion-free divisible abelian groups are pseudofinite\footnote{For instance $\prod_{p\in\U} (\Z/p\Z,+)$ is certainly a $\Q$-vector space (a $p$-group is $n$-divisible for all $n$ coprime to $p$).}. 
Note that being divisible is a different condition than $\ker\theta \neq \vect{\ker\theta}^\dv$, which always holds in a model of ACFH (see Example \ref{rk_ker_notrootclosed}). Note also that $K/G\cong K$ is not incompatible with $G$ being itself a $\Q$-vector space: take $K = \mu_{p^\infty}\odot \Q$ and $G = \Q$, then $K/G$ has the same torsion as $K$.
\end{example}
\begin{lemma}\label{lm_herzog_ineq2}
Let $D$ be a divisible abelian group with finite $n$-torsion for all $n$. Let $\phi:D\rightarrow D$ be a surjective homomorphism and $G = \ker \phi$. Then
 $\abs{G/\pw^n(G)}= \abs{\mu_n(G)}$, for all $n\in \N$.
\end{lemma}
\begin{proof}
We have $D/G\cong D$, hence $\mu_n(D/G) = \mu_n(D)$, and so  $\abs{G/\pw^n(G)}\leq \abs{\mu_n(G)}$ by Lemma \ref{lm_herzog_inequ1}. It suffices to show that $\abs{G/\pw^n(G)}\geq \abs{\mu_n(G)}$. Let $\phi_{0} = \phi\upharpoonright\mu_n(D):\mu_n(D)\rightarrow \phi(\mu_n(D))$. Then $\ker\phi_{0} = \mu_n(D)\cap \ker\phi = \mu_n(G)$. This means that $\phi(\mu_n(D))\cong \mu_n(D)/\mu_n(G)$. Let $N = \abs{\mu_n(D)}$, $M = \abs{\mu_n(G)}$ and $k=N/M$. Then $\abs{\phi(\mu_n(D))} = N/M = k$ and hence \[\abs{\mu_n(D)/\phi(\mu_n(D))} = N/k = M.\]
Let $\xi_1,\ldots,\xi_M$ be representatives of $\mu_n(D)$ modulo $\phi(\mu_n(D))$. As $\phi$ is surjective, there exist $b_1,\ldots,b_M\in D$ such that $\phi(b_i) = \xi_i$.

Let $a_i = b_i^n$, we have $\phi(a_i) = 1$, hence $a_i\in G = \ker\phi$. We prove that for $i\neq j$, $a_i$ and $a_j$ are in different classes modulo $\pw^n(G)$. Assume that $i\neq j$ and by contradiction $a_ia_j^{-1}\in \pw^n(G)$. Then there exists $g\in G$ such that $(b_ib_j^{-1})^n = g^n$, hence there exists $\zeta\in \mu_n(D)$ such that $b_ib_j^{-1} = \zeta g$. Applying $\phi$, we get $\xi_i\xi_j^{-1} = \phi(\zeta)$ (since $\phi(g)=1$), and so $\xi_i\xi_j^{-1}\in \phi(\mu_n(D))$, a contradiction. 
\end{proof}

\begin{theorem}\label{thm_ker_surj_endo_psfc}
Let $K$ be any field with $K^\times$ divisible. Let $\phi:K^\times\rightarrow K^\times$ be a surjective endomorphism. Then $\ker\phi$ is a pseudofinite-cyclic group.
\end{theorem}
\begin{proof}
Apply Fact \ref{fact_herzog} with Lemma \ref{lm_herzog_ineq2}, as $\mu_n(K^\times )\leq n$, for any field.
\end{proof}

Combining Theorem \ref{thm_ker_surj_endo_psfc} and Corollary \ref{cor_def_endo_surj}, we conclude the following.
\begin{corollary}\label{cor_kernelspseudofiniteACFH}
Let $(K,\theta)\models \ACFH$. Then $\ker(P(\theta))$ is pseudofinite-cyclic as a pure group, for all $P(X)\in \Z[X]$.
\end{corollary}


\subsection{Nonstandard power functions are not natural models of ACFH}\label{sub_natural_model}

We end this paper with considerations on the search for \textit{natural models} of ACFH. Here, \textit{natural} should be understood as somehow \textit{explicitely constructed} models, although \textit{explicitely} (and \textit{constructed}) here has to be considered relaxed to model theorists' standards. For instance, it is the author's opinion that the ultraproduct construction yields natural objects, provided the factors are explicit enough, a claim which is fairly debatable. This opinion is mostly inspired by classical beautiful results such as: a non-principal ultraproduct of finite fields is a model of the theory of pseudo-finite fields, a non-principal ultraproduct of $(\F_p^\alg, \Frob)$ is a model of ACFA. Of course those results are the terrain for connections and applications of model theory to other fields, and that is the underlying motivation behind this sort of questions. 

\begin{question}
Is there a family $(K_i,\theta_i)_{i\in I}$ of models of $T$ (which are not models of ACFH) such that for some ultraproduct $\U$ on $I$, the ultraproduct $\prod_\U (K_i,\theta_i)$ is a model of ACFH?
\end{question}

Let $\PP$ be the set of prime numbers and let $(s_p)_{p\in \PP}\in \N^\PP$. For each $p\in \PP$, the map $\pw^{s_p}:x\mapsto x^{s_p}$ is a multiplicative (surjective) map $\F_p^\alg\rightarrow \F_p^\alg$, hence $(\F_p^\alg,\pw^{s_p})$ is a model of $T$, for all $p\in \PP$.
Let $\U$ be any ultrafilter on $\PP$, and consider 
\[(\C,\theta) := \prod _\U  (\F_p^\alg,\pw^{s_p})\]
the ultraproduct of $(\F_p^\alg,\pw^{s_p})$ along $\U$. It is clear that $\ker(\theta)$ is pseudofinite. More generally $\ker P(\theta)$ is also pseudo-finite for $P\in \Z[X]$, since $P(\pw^{s_p})$ is a power function. A natural question to ask is the following:
 \begin{center}
    Is there a choice of $(s_p)_{p\in \PP}$ for which $(\C,\theta)$ is a model of ACFH?
\end{center}
The answer is no. A first intuitive reason is the following simple observation: $\pw^n$ is a multiplicative map that sends m-generics to m-generics, in particular, $\pw^n$ sends multiplicatively independent tuples to multiplicatively independent tuples (even though it is not injective). 
This should also hold in the ultraproduct (reasonning as in \cite[Lemma 3.10]{dEKN21}), but is not true in a model of ACFH, where the kernel of $\theta$ has infinite multiplicative dimension. We give a less informal (and simpler) argument.
\begin{lemma}\label{lm_ultraproduct_model}
Assume that $(K_i,\theta_i)_{i\in I}$ is a family of models of $T$ such that for some non-principal ultrafilter $\U$ on $I$, the ultraproduct $\prod_\U (K_i,\theta_i)$ is a model of ACFH. Then for almost all $i\in I$, $K_i = \ker\theta_i+\ker\theta_i$. In particular $\ker\theta_i$ is infinite, for almost all $i\in I$.
\end{lemma}
\begin{proof}
This is true in any model of ACFH, by Example \ref{ex_stablyembedded} and it is expressible as a first order property, so it must hold for almost all $(K_i,\theta_i)$.
\end{proof}

Regardless of the choice of $s_p\neq 0$, the kernel of $\pw^{s_p}$ is finite, hence for all $(s_p)_{p\in \PP}\in (\N\setminus\set{0})^\PP$ $(\C,\theta)$ is not a model of ACFH. We are facing an unintuitive phenomenon: for any model $(K,\theta)$ of ACFH, $\ker\theta$ is pseudofinite as a pure group, however there is no family $(K_i,\theta_i)_{i\in I}$ of models of $T$ with $\ker\theta_i$ finite such that $(K,\theta)\equiv \prod_\U (K_i,\theta_i)$ for some ultrafilter $\U$ on $I$. 

\begin{remark}\label{rk_towardtheMTofCtheta}
 There is room for several variants of the construction of ACFH. One may add to the theory $T$ the condition `$\theta$ is injective'. In this case, the kernel of $\theta$ vanishes, but not necessarily the kernels of other definable endomorphisms, such as $x\mapsto x^{-1}\theta(x)$ (i.e. the subgroup of elements fixed by $\theta$). One could ask for all kernels to vanish at the same time, and hope that a model-companion exists, with $\ker P(\theta) = \set{1}$ for all $P\in \Z[X]$. 

In order to detect the genericity of the structure $(\C,\theta)$, one might consider another variant of the construction of ACFH, which comprise the condition `$\theta$ sends m-generics to m-generics'. We keep the same theory $T$, but we restrict the notion of extensions of models of $T$: $(K,\theta_K)\seq (L,\theta_L)$ is called \textit{kernel-preserving} if $\ker\theta_K = \ker\theta_L$. Then a similar approach as in Section \ref{sec:ecmodels} should yield: $(K,\theta)$ is existentially closed in every kernel-preserving extension if and only if  for all m-varieties $\tau\seq K^n$ over $K$ and every affine variety $V\seq \tau\times \tau^\theta$ which projects m-generically onto $\tau$ \textit{and onto $\tau^\theta$}, there exists a tuple $a\in K^n$ such that $(a,\theta(a))\in V$. Let $\CC$ be the class of models of $T$ which are existentially closed in every kernel-preserving extension. It should be clear that this class is not elementary in general, but one should be able to consider some particular cases. Let $T_{G_0}$ be the extension of $T$ by constants for a subgroup $G_0$ (of the source of $\theta$) and expressing that the kernel of $\theta$ contains $G_0$. Let $\CC_{G_0}$ be the class of models of $T_{G_0}$ which are existentially closed in every kernel-preserving extension. It should be true that for elements of $\CC_{G_0}$, the kernel of $\theta$ is always $G_0$. If $G_0$ is finite, then this class is the same as existentially closed models of the theory $T$ expanded by the (first order) condition `$\ker \theta = G_0$'. This class should be first-order axiomatizable, using similar methods as those in the present paper. If $G_0$ is infinite, then $G_0$ should be an infinite definable set of bounded size, hence the class $\CC_{G_0}$ should not be elementary axiomatizable. Positive logic might be promising for studying the class $\CC_{G_0}$ for $G_0$ infinite.

\end{remark}

\appendix

\section{Proof of Proposition \ref{prop_ext_mult}}\label{appx_proofcms}

We now give a full proof of Proposition \ref{prop_ext_mult}.

\begin{lemma}\label{lm_ordersingleelement_mult}
Let $C = \vect C ^\dv\seq K^\times$, $b\in K$ and $a=(a_1,\dots,a_r)\in (K^\times)^r$ be multiplicatively independent over $C$. Assume that for some $m\in \N\setminus \set{0}$ and $l_1,\dots,l_r\in \N$ we have $b^m = ca_1^{l_1}\ldots a_r^{l_r}$. Then $\gcd(m,l_1,\dots,l_r) = 1$ if and only if $m$ is the smallest $n>0$ such that $b^n\in \vect{Ca}$. 
\end{lemma}
\begin{proof}
If $b^n = ca_1^{l_1}\ldots a_r^{l_r}$, $\gcd(n,l_1,\dots,l_r) = 1$ and $n$ is not minimal, let $m$ be the minimal $m>0$ such that $b^m\in \vect{Ca}$. We have $m<n$. Consider the euclidean division of $n$ by $m$, minimality of $m$ implies that $m$ divides $n$, say $mq = n$. If $b^m = c'a_1^{m_1}\ldots a_r^{m_r}$, then we get $1 = cc'^{-q}a_1^{l_1-qm_1}\ldots a_r^{l_r-qm_r}$. As $a$ is multiplicatively independent over $C$, we get $l_i-qm_i = 0$ hence $q$ divides each $l_i$, a contradiction. 

Conversely, if $n = \min\set{m\mid b^m\in \vect{Ca}}$, let $d = \gcd(n,l_1,\ldots,l_r)$. Let $n = dm$ and $l_i = dk_i$. As $C = \vect{C}^\dv$, there exists $c'$ such that $c'^d = c$. We get 
\[(b^{m})^d = (c'a_1^{k_1}\ldots a_r^{k_r})^d.\]
As $C$ contains the elements of $d$-torsion, it follows that there is $c''\in C$ such that $b^{m} = c''a_1^{k_1}\ldots a_r^{k_r}$ so $b^{m}\in \vect{Ca}$. By minimality of $n$, we have $n = m$ hence $1= d = \gcd(n,l_1,\ldots,l_r)$.
\end{proof}

Assuming $b\in \vect{Ca}^\dv$. If $m\in\N$ is the order of $b$ over $\vect{Ca}$, then $\set{n\in \N\mid b^n\in \vect{Ca}} = m\N$. If $b^m = ca_1^{l_1}\ldots a_r^{l_r}$, then \[b^\Z\cap \vect{Ca} = b^{m\Z} = \set{c^na_1^{nl_1}\ldots a_r^{nl_r}\mid n\in \Z}.\]

\begin{lemma}\label{lm_ext_mult}
Let $C = \vect C ^\dv, a = (a_1,\dots,a_r)\in (K^\times)^r, b = (b_1,\dots, b_t)\in (K^\times)^t$. Assume that $a$ is multiplicatively independent over $C$. 
\begin{enumerate}
    \item $b\subseteq \vect{Ca}^\dv$ if and only if there is a unique complete system of minimal equations $\tau(x;y)$ over $C$ such that $K\models \tau(a;b)$.
    \item Let $\theta: C\rightarrow C'$ be any multiplicative homomorphism, and let $a',b'\subseteq K$ be such that $K\models \tau(a;b)\wedge \tau^\theta(a';b')$, for some complete system of minimal equations $\tau(x;y)$ over $C$. Then there is a multiplicative homomorphism $\theta': \vect{Cab}\rightarrow \vect{C'a'b'}$ extending $\theta$ and such that $\theta(ab) = a'b'$.
\end{enumerate}
\end{lemma}
\begin{proof}
(1) From right to left, if there is a complete system of minimal equations $\tau(x;y)$ over $C$ such that $K\models \tau(a;b)$, then in particular, $b_i$ satisfy equations of the form $b_i^{n_i} = ca_1^{k_1}\ldots a_t^{k_t}$ for $n_i\in \N\setminus\set{0}$, so $b\seq \vect{Ca}^\dv$. For the other direction, assume $b\subseteq \vect{Ca}^\dv$. Let $n_i$ be the order of $b_i$ over $\vect{Ca}$ and let $\CC = \set{(k_1,\dots,k_t)\mid 0\leq k_i\leq n_i \ \gcd(k_1,\ldots,k_t) = 1}$. For each $(k_1,\dots,k_t)\in \CC$, $b_1^{k_1}\ldots b_t^{k_t}$ is of finite order over $\vect{Ca}$ (bounded by $n_1\ldots n_t$), hence by Lemma \ref{lm_ordersingleelement_mult}, there exists $c\in C$, $N\in \N$ and $l_1,\ldots,l_r\in \Z$ such that $\gcd(N,l_1,\ldots,l_r) = 1$ and ${(b_1^{k_1}\ldots b_t^{k_t})}^N = ca_1^{l_1}\ldots a_r^{l_r}$. For $(k_1,\ldots,k_r) = (0\ldots,1,\ldots0)$ ($1$ at the $i$-th position), it yields $b_i^N= ca_1^{l_1}\ldots a_r^{l_r}$ so by Lemma \ref{lm_ordersingleelement_mult}, $N = n_i$. Let $\tau(x;y)$ be the cunjunction of each equation for each $(k_1,\ldots,k_t)\in \CC$, $\tau(x;y)$ is a complete system of minimal equations over $C$ and $K\models \tau(a;b)$. It remains to show uniqueness of such system. Assume that $\tau'(x;y)$ is another complete system of minimal equation over $C$ satisfied by $(a;b)$, say associated to $n_1',\ldots,n_t'$, $\CC'$ etc. By definition and Lemma \ref{lm_ordersingleelement_mult}, $n_i'$ is the order of $b_i$ over $\vect{Ca}$ hence $n_i' = n_i$ and $\CC' = \CC$. For a given $(k_1,\ldots,k_t)\in \CC$, let ${(b_1^{k_1}\ldots b_t^{k_t})}^N = ca_1^{l_1}\ldots a_r^{l_r}$ be the corresponding equation in $\tau$ and $(b_1^{k_1}\ldots b_t^{k_t})^{N'} = c'a_1^{l_1'}\ldots a_r^{l_r'}$ be the corresponding equation in $\tau'$. As $\gcd(N,l_1,\ldots,l_r) = \gcd(N',l_1',\ldots,l_r') = 1$, by Lemma \ref{lm_ordersingleelement_mult}, $N = N'$ is the order of $b_1^{k_1}\ldots b_t^{k_t}$ over $\vect{Ca}$, and as $a$ is multiplicatively independent over $C$, it also follows that $l_i = l_i'$ and $c = c'$. We conclude that $\tau = \tau'$.

(2) As $a$ is multiplicatively independent over $C$, it is clear that $\theta$ extends to an homomorphism $\theta_1:\vect{Ca}\rightarrow \vect{Ca'}$ by sending $a_i\mapsto a_i'$. More precisely, as $a$ is multiplicatively independent over $C$, we can write $\vect{Ca} = C\odot a_1^\Z \odot\ldots \odot a_r^\Z$. 
Now for any choice of $a_i'$, the map $a_i^k\mapsto a_i'^k$ defines a group homomorphism from $a_i^\Z$ onto $a_i'^\Z$ (note that if the order of $a_i'$ is infinite, this is an isomorphism, otherwise, the map is not injective and the map goes onto the finite group $\mu_n$ for $n$ the order of $a_i'$)\footnote{There is a more general condition for this: let $a,b$ be singletons, then there is a group homomophism $h:a^\Z\rightarrow b^\Z$ \textit{sending $a$ on $b$} if and only if the order of $b$ divides the order of $a$, with the convention that any $n$ divides infinity and infinity divides infinity.}. In turn we have a map $\theta_0:\vect{a}\rightarrow \vect{a'}\seq \vect{Ca'}$ and a map $\theta:C\rightarrow C'\seq \vect{Ca'}$, so we can define the map $\theta_1: \vect{Ca}\rightarrow \vect{Ca'}$ by $\theta_1(cu) = \theta(c)\theta_0(u)$ for $c\in C$ and $u\in \vect{a}$. This is a well defined map because $C\cap\vect{a} = \set{1}$ so the decomposition of every element of $\vect{Ca}$ into $cu$ is unique. It is easy to check that as defined, the map $\theta_1$ is indeed a homomorphism. In order to show that we can extend $\theta_1$ to an homomorphism $\theta':\vect{Cab}\rightarrow \vect{Ca'b'}$ which maps $b_i$ to $b_i'$, one has to check that if $b_1^{m_1}\ldots b_t^{m_t} = c a_1^{u_1}\ldots a_r^{u_r}$, then $b_1'^{m_1}\ldots b_t'^{m_t} = \theta(c) a_1'^{u_1}\ldots a_r'^{u_r}$, for all $c\in C$, $m_i,u_i\in \Z$. (Indeed, you can then define $\theta'(b_i) = b_i'$ and extend linearly to a map $\vect{Cab}\rightarrow \vect{C'a'b'}$ by setting $\theta'(ca_1^{l_1}\ldots a_r^{l_r}b_1^{k_1}\ldots b_t^{k_t}) = \theta(c)\theta_1(a_1)^{l_1}\ldots \theta_1(a_r)^{l_r}\theta'(b_1)^{k_1}\ldots \theta'(b_t)^{k_t} = \theta(c)a_1'^{l_1}\ldots a_r'^{l_r}b_1'^{k_1}\ldots b_t'^{k_t}$. To check that this is a well defined map, one has to check that if \[ca_1^{l_1}\ldots a_r^{l_r}b_1^{k_1}\ldots b_t^{k_t} = c'a_1^{l_1'}\ldots a_r^{l_r'}b_1^{k_1'}\ldots b_t^{k_t'}\] then
\[\theta'(ca_1^{l_1}\ldots a_r^{l_r}b_1^{k_1}\ldots b_t^{k_t}) = \theta'(c'a_1^{l_1'}\ldots a_r^{l_r'}b_1^{k_1'}\ldots b_t^{k_t'}).\]
So it is enough to prove that if $ca_1^{l_1}\ldots a_r^{l_r}b_1^{k_1}\ldots b_t^{k_t} = 1$, then $\theta'(ca_1^{l_1}\ldots a_r^{l_r}b_1^{k_1}\ldots b_t^{k_t}) = 1$, which is equivalent to proving that if $b_1^{m_1}\ldots b_t^{m_t} = c a_1^{u_1}\ldots a_r^{u_r}$, then $b_1'^{m_1}\ldots b_t'^{m_t} = \theta(c) a_1'^{u_1}\ldots a_r'^{u_r}$, for all $c\in C$, $m_i,u_i\in \Z$.)

   Assume that $b_1^{m_1}\ldots b_t^{m_t} = c a_1^{u_1}\ldots a_r^{u_r}$. For each $i\leq t$, let $m_i = n_iq_i+r_i$ be the euclidean division of $m_i$ by $n_i$, for $n_i$ the order of $b_i$ over $\vect{Ca}$. We get 
\[b_1^{r_1}\ldots b_t^{r_t} = c a_1^{u_1}\ldots a_r^{u_r}b_1^{-n_1q_1}\ldots b_t^{-n_t q_t}.\]
As $b_i^{-n_iq_i}\in \vect{Ca}$, there is $c_1\in C$ and $v_1,\ldots,v_r\in \Z$ such that $b_1^{r_1}\ldots b_t^{r_t} = c_1 a_1^{v_1}\ldots a_r^{v_r}$. Let $d = \gcd(r_1,\ldots,r_t,v_1,\ldots,v_r)$, $r_i = ds_i$, $v_i = d l_i$. As $C = \vect{C}^\dv$, there is $c_2$ such that $c_2^d = c_1$. We have 
\[(b_1^{s_1}\ldots b_t^{s_t})^d = (c_2 a_1^{l_1}\ldots a_r^{l_r})^d.\]
As $C = \vect{C}^\dv$, there is $c_3$ such that $c_3^d = c_1$ and
\[b_1^{s_1}\ldots b_t^{s_t} = c_3 a_1^{l_1}\ldots a_r^{l_r}.\]
Let $N = \gcd(s_1,\ldots,s_t)$ and $s_i = Nk_i$. As $\gcd(s_1,\ldots,s_t,l_1\ldots,l_r) = 1$, we have $\gcd(N,l_1,\ldots,l_r) = 1$ and 
\[(b_1^{k_1}\ldots b_t^{k_t})^N = c_3a_1^{l_1}\ldots a_r^{l_r}.\]
The latter is an instance of a minimal equation satisfied by $b_1,\ldots ,b_t,a_1,\ldots,a_r$. By uniqueness of the complete system of minimal equations (1), $(y_1^{k_1}\ldots y_t^{k_t})^N = c_3x_1^{l_1}\ldots x_r^{l_r}$ is in $\tau(x;y)$, hence by hypotheses, $(b_1'^{k_1}\ldots b_t'^{k_t})^N = \theta(c_3)a_1'^{l_1}\ldots a_r'^{l_r}$. 
It remains to check that $b_1'^{m_1}\ldots b_t'^{m_t} = \theta(c) a_1'^{u_1}\ldots a_r'^{u_r}$. From $(b_1'^{k_1}\ldots b_t'^{k_t})^N = \theta(c_3)a_1'^{l_1}\ldots a_r'^{l_r}$, we get $b_1'^{s_1}\ldots b_t'^{s_t} = \theta(c_3)a_1'^{l_1}\ldots a_r'^{l_r}$. By raising to the $d$-th power, we have $b_1'^{r_1}\ldots b_t'^{r_t} = (\theta(c_3))^d a_1'^{v_1}\ldots a_r'^{v_r}$. As $(\theta(c_3))^d = \theta(c_3^d) = \theta(c_1)$, we have 
\[b_1'^{r_1}\ldots b_t'^{r_t} = \theta(c_1) a_1'^{v_1}\ldots a_r'^{v_r}.\quad (\star)\]  

We have $ca_1^{u_1}\ldots a_r ^{u_r}b_1^{-n_1q_1} \ldots b_t^{-n_tq_t} = c_1a_1^{v_1}\ldots a_r^{v_r}\in \vect{Ca}$. We may apply $\theta_1$ on both sides to get $\theta_1(ca_1^{u_1}\ldots a_r ^{u_r}b_1^{-n_1q_1} \ldots b_t^{-n_tq_t}) = \theta(c_1)a_1'^{v_1}\ldots a_r'^{v_r}$. Also, \[\theta_1(ca_1^{u_1}\ldots a_r ^{u_r}b_1^{-n_1q_1} \ldots b_t^{-n_tq_t}) = \theta(c)a_1'^{u_1}\ldots a_r'^{u_r}\theta_1(b_1^{n_1})^{-q_1}\ldots \theta_1(b_t^{n_t})^{-q_t}.\quad (\ast)\]
Note that $b_i^{n_i}\in \vect{Ca}$, but not $b_i$, hence $\theta_1(b_i^{n_i})$ makes sense but not $\theta_1(b_i)^{n_i}$. Now recall that if $b_1^{n_1} = c_0a_1^{w_1}\ldots a_r^{w_r}$, then we also have $b_1'^{n_1} = \theta(c_0)a_1'^{w_1}\ldots a_r'^{w_r}$. This follows from the very definition of a complete system of minimal equations (Definition \ref{def:completesystemofminimalequations}): $N = n_i$ for $(k_1,\ldots,k_t) = (0,\ldots,1,\ldots,0)$. As $\theta_1:\vect{Ca}\to \vect{Ca'}$ we conclude that $\theta_1(b_1^{n_1}) = b_1'^{n_1}$ and more generally $\theta_1(b_i^{n_i}) = b_i'^{n_i}$. It follows from $(\ast)$ that \[\theta(c_1)a_1'^{v_1}\ldots a_r'^{v_r}= \theta(c)a_1'^{u_1}\ldots a_r'^{u_r}b_1'^{-n_1q_1}\ldots b_t'^{-n_tq_t}.\]
Using $(\star)$ we get $b_1'^{m_1}\ldots b_t'^{m_t} = \theta(c) a_1'^{u_1}\ldots a_r'^{u_r}$, which concludes the proof.
\end{proof}

To get Proposition \ref{prop_ext_mult} (1): for any $a$ and $C = \vect{C}^\dv$, there is a maximal subtuple $a'\seq a$ multiplicatively independent over $C$, then $b:=a\setminus a'\seq \vect{Ca'}^\dv$, so apply Lemma \ref{lm_ext_mult} (1). By definition, Proposition \ref{prop_ext_mult} (2) is just Lemma \ref{lm_ext_mult} (2).

\section{Pseudofinite-cyclic groups (joint with I. Herzog)}\label{appx_psfc}

This section is joint with I. Herzog.

We switch to additive notation for abelian groups. In this section, we use abundantly classical results of Szmielew about the model theory of abelian groups, as presented in \cite[Appendix A.2]{Hodges}. For an abelian group $G$ and a prime $p$, we denote $G[p] = \set{g\in G\mid pg = 0}$ and $pG = \set{pg\mid g\in G}$.

\begin{proposition}\label{prop_herzog}
An abelian group $G$ is pseudofinite-cyclic if and only if for all primes $p$ we have:
\[\abs{G[p]} = \abs{G/pG}\leq p.\]
\end{proposition}

\begin{proof}
Assume that $G$ is pseudofinite-cyclic. Note that for any finite cyclic group $C$ and prime $p$, the map $x\mapsto px$ induces an isomorphism between $C/C[p] $ and $pC$. It follows that $\frac{\abs{C}}{\abs{pC}} = \abs{C[p]}$ hence $\abs{C/pC} = \abs{C[p]}$. As $C$ is cyclic, $C[p]$ is cyclic of order $1$ or $p$, hence we get $\abs{C[p]} = \abs{C/pC}\leq p$. This is an elementary statement hence holds for all pseudofinite-cyclic groups.

Conversely, assume that $G$ is a group satisfying $\abs{G[p]} = \abs{G/pG}\leq p$, for all primes $p$. We call this condition $(\star)$. By quantifier elimination for abelian groups, $G$ is elementary equivalent to a Szmielew group, we may assume that $G$ is of the form
\[\bigoplus_{p\in \PP} \left[ \oplus_{n>0}\Z(p^n)^{\kappa_{p,n}}\oplus \Z(p^\infty)^{\lambda_p}\oplus \Z_p^{\nu_p}\right]\oplus \Q^\epsilon\]
    with $\kappa_{p,n}, \lambda_p, \nu_p$ finite or countable, and $\epsilon=0$ or $\omega$. Let $G_p$ be the $p$-component of the torsion subgroup of $G$ so that $G[tor] = \bigoplus_{p} G_p$, i.e. $G_p = \oplus_{n>0}\Z(p^n)^{\kappa_{p,n}}\oplus \Z(p^\infty)^{\lambda_p}$, and let $H_p = G_p\oplus \Z_p^{\nu_p}$
    \begin{claim}
    For all $p\in \PP$, $H_p = \Z(p^{\alpha_p})$ or $H_p = \Z(p^\infty)\oplus \Z_p$.
    \end{claim}
    \begin{proof}[Proof of the claim]
    Observe first that $H_p$ satisfy $(\star)$. If $G_p $ is finite, we prove that $G_p = \Z(p^\alpha_p)$ and $\mu_p = 0$. If $G_p$ is finite, then $G_p = \Z(p^{\alpha_1})\oplus \ldots \oplus \Z(p^{\alpha_n})$. Such group satisfies $\abs{G_p[p]}\leq  p$ if and only if $n=1$ hence $G_p = \Z(p^{\alpha_p})$. If $G_p = 0$, then $H_p = \Z_p^{\nu_p}$. As $H_p[p] = 0$, it follows from $(\star)$ that $pH_p = H_p$ hence $\nu_p = 0$. If $G_p = \Z(p^{\alpha_p})$ with $\alpha_p\neq 0$, then $\abs{G_p/pG_p} = p$. It follows that $(G_p\oplus \Z_p^{\nu_p})/p(G_p \oplus \Z_p^{\nu_p})$ is of order $p^{1+\nu_p}$, hence $(\star)$ implies that $\nu_p = 0$.
 
 If $G_p$ is infinite, then $G_p = \Z(p^\infty)$ and $\mu_p = 1$. If $G_p$ is infinite then $\lambda_p\geq 1$.  As $\abs{G_p[p]}\leq p$, there are not factors of the form $\Z(p^{\alpha_1})\oplus \ldots \oplus \Z(p^{\alpha_n})$, and $\lambda_p = 1$ hence $G_p = \Z(p^\infty)$. As $H_p = \Z(p^\infty)\oplus \Z_p^{\nu_p}$, $\nu_p = 0$ implies that $pH_p = H_p$ and contradicts $\abs{H_p/pH_p} = H_p[p] = p$ and $\nu_p\geq 1$ implies that  $H_p/pH_p$ is of order $p^{\nu_p}$ hence by $(\star)$ we have $\nu_p = 1$.
    \end{proof}
    
    Let $P = \set{p\in \PP\mid H_p = \Z(p^{\alpha_p})}$ and $Q = \set{p\in \PP\mid H_p = \Z(p^\infty)\oplus \Z_p}$. Then $G$ is elementary equivalent to 
    \[\bigoplus_{p\in P} \Z(p^{\alpha_p}) \oplus \bigoplus_{q\in Q} (\Z(q^\infty)\oplus \Z_q)\oplus \Q^\epsilon.\]
    If both $P$ and $Q$ are empty, then $G$ is elementary equivalent to $\Q^\epsilon$, which is easily checked to be elementary equivalent to the ultraproduct of $\Z(p)$ where $p$ varies. So we may assume that $P$ or $Q$ are nonempty.
We define cyclic groups $(F_n)_{n< \omega}$.
 \begin{itemize}
    \item If $P$ is infinite. Let $(p_n)_{n<\omega}$ be an enumeration of $P$ and for each $n<\omega$, we define:
    \[F_n = \Z(p_1^{\alpha_1})\oplus\ldots\oplus \Z(p_n^{\alpha_n}).\]
    \item If $P$ is finite and $Q\neq \emptyset$, we define $F_n = \bigoplus_{p\in P} \Z(p^{\alpha_p})$, for all $n<\omega$. 
    \item If $P$ is finite and $Q = \emptyset$, we define $F_n = \bigoplus_{p\in P} \Z(p^{\alpha_p})\oplus \Z(s_1^n)\oplus\ldots \oplus \Z(s_n^n)$, for all $n<\omega$, for some enumeration $(s_n)_{n<\omega}$ of $\PP\setminus P$.
 \end{itemize}
    Then, for any non-principal ultrafilter $\U$ on $\omega$, one checks that:
\[\prod_{n<\omega} F_n/\U \equiv \oplus_{p\in P} \Z(p^{\alpha_p})\oplus \Q^{\epsilon_0}\]
where $\epsilon_0 = 0$ if $P$ is finite and $Q = \emptyset$, and $\epsilon_0 = \omega$ if $P$ is infinite or $P$ is finite and $Q = \emptyset$. (To check this, if $P$ is finite, this is clear, so assume $P$ infinite. Observe first that the torsion $N[tor]$ of $N=\prod_{n<\omega} F_n/\U$ is $\oplus_{p\in P} \Z(p^{\alpha_p})$, and that $N/N[tor]$ is divisible and torsion-free, so isomorphic to a $\Q$-vector space. Then use the classical fact that for abelian groups $A\equiv B$ and $C\equiv D$ implies $A\oplus B\equiv C\oplus D$, see \cite[Lemma A.1.6]{Hodges}.)

We define cyclic groups $(K_n)_{n<\omega}$.
\begin{itemize}
\item If $Q$ is infinite, let $(q_n)_{n<\omega}$, be an enumeration of $Q$ and we define for each $n<\omega$:
\[K_n = \Z(q_1^n)\oplus\ldots\oplus \Z(q_n^n).\]
\item If $Q$ is finite, define $K_n = \bigoplus_{q\in Q} \Z(q^n)$.
\end{itemize}
 Note that in both cases $K_n$ is cyclic for all $n<\omega$. Then, for any non-principal ultrafilter $\U$ on $\omega$, one checks that 
\[\prod_{n<\omega} K_n/\U \equiv \bigoplus_{q\in Q} \Z(q^\infty)\oplus \Z_q.\]
(Again, to check this, let $M = \prod_{n<\omega} K_n/\U$. First, identify the torsion $M[tor] = \oplus_{q\in Q} \Z(q^\infty)$ and observe that the torsion free-part $L = M/M[tor]$ is $n$-divisible for $n$ coprime to $Q$, and that $qL$ is of index $q$ in $L$, for all $q\in Q$.)

Observe that, as $P\cap Q = \emptyset$, $F_n\oplus K_n$ is finite cyclic. Then one checks that 
\[\prod_{n<\omega} F_n\oplus K_n/\U \equiv \bigoplus_{p\in P} \Z(p^{\alpha_p}) \oplus \bigoplus_{q\in Q} (\Z(q^\infty)\oplus \Z_q)\oplus \Q^{\epsilon_0}\]

Finally, we have to check that we may assume $\epsilon_0 = \epsilon$. If $Q \neq \emptyset$ or $P$ is infinite, then $G$ has unbounded exponent, hence we may assume $\epsilon = 0$ and also $\epsilon_0 = 0$. If $Q = \emptyset$ and $P$ is finite, $\epsilon = \omega$ since $G$ is infinite. Also, by choice,  $\epsilon_0 = \omega$ hence $\epsilon = \epsilon_0$.

We conclude that $G$ is elementary equivalent to $\prod_{n<\omega} F_n\oplus K_n$, so $G$ is pseudofinite-cyclic.
\end{proof}

\bibliographystyle{alpha}
\bibliography{biblio}

\end{document}